\documentclass[11pt]{article}
\usepackage{amsmath,amssymb,amsthm}

\newtheorem{theorem}{Theorem}[section]
\newtheorem{lemma}[theorem]{Lemma}
\newtheorem{proposition}[theorem]{Proposition}
\newtheorem{corollary}[theorem]{Corollary}

\theoremstyle{definition}
\newtheorem{definition}[theorem]{Definition}

\numberwithin{equation}{section}

\renewcommand{\labelenumi}{\textup{(\theenumi)}}

\renewcommand{\phi}{\varphi}

\newcommand{\Homeo}{\operatorname{Homeo}}
\newcommand{\id}{\operatorname{id}}
\newcommand{\Ker}{\operatorname{Ker}}

\newcommand{\N}{\mathbb{N}}
\newcommand{\Z}{\mathbb{Z}}

\title{Continuous orbit equivalence of topological Markov shifts 
and dynamical zeta functions}
\author{Kengo Matsumoto \\
Department of Mathematics \\
Joetsu University of Education \\
Joetsu, Niigata 943-8512, Japan
\and
Hiroki Matui \\
Graduate School of Science \\
Chiba University \\
Inage-ku, Chiba 263-8522, Japan}
\date{}

\begin{document}
\maketitle

\def\det{{{\operatorname{det}}}}

\begin{abstract}
For continuously orbit equivalent
 one-sided topological Markov shifts $(X_A,\sigma_A)$
and $(X_B,\sigma_B)$,
 their eventually periodic points
and cocycle functions are studied.
As a result
we  directly construct an isomorphism
between their ordered cohomology groups
$(\bar{H}^A, \bar{H}^A_+)$ and
$(\bar{H}^B, \bar{H}^B_+)$.
We also show that the cocycle functions for the 
continuous orbit equivalences 
give rise to positive elements of  
the ordered cohomology,
so that 
the the zeta functions 
of continuously orbit equivalent topological Markov shifts
are related.
The set of Borel measures 
is shown to be invariant under 
continuous orbit equivalence of one-sided topological Markov shifts.
\end{abstract}




\def\Z{{ {\mathbb{Z}} }}
\def\N{{ {\mathbb{N}} }}
\def\Zp{{ {\mathbb{Z}}_+ }}
\def\FKL{{ {\cal F}_k^{l} }}
\def\FKI{{{\cal F}_k^{\infty}}}
\def\FLI{{ {\cal F}_{\Lambda} }}
\def\SES{{S_{\mu}E_l^iS_{\mu}^*}}
\def\AL{{{\cal A}_{\Lambda}}}
\def\OL{{{\cal O}_{\Lambda}}}
\def\OA{{{\cal O}_A}}
\def\BOA{{\Bar{\cal O}_A}}
\def\BOB{{\Bar{\cal O}_B}}
\def\OB{{{\cal O}_B}}
\def\FA{{{\cal F}_A}}
\def\FB{{{\cal F}_B}}
\def\DA{{{\frak D}_A}}
\def\BDA{{\Bar{\frak D}_A}}
\def\BDB{{\Bar{\frak D}_B}}
\def\DB{{{\frak D}_B}}
\def\HA{{{\frak H}_A}}
\def\HB{{{\frak H}_B}}
\def\Ext{{{\operatorname{Ext}}}}
\def\Per{{{\operatorname{Per}}}}
\def\PerB{{{\operatorname{PerB}}}}
\def\Homeo{{{\operatorname{Homeo}}}}
\def\HSA{{H_{\sigma_A}(X_A)}}
\def\Out{{{\operatorname{Out}}}}
\def\Aut{{{\operatorname{Aut}}}}
\def\Inn{{{\operatorname{Inn}}}}
\def\det{{{\operatorname{det}}}}
\def\cobdy{{{\operatorname{cobdy}}}}
\def\Ker{{{\operatorname{Ker}}}}
\def\ind{{{\operatorname{ind}}}}
\def\id{{{\operatorname{id}}}}
\def\supp{{{\operatorname{supp}}}}


\section{Introduction}

Let $A$ be an irreducible square matrix with entries in $\{0,1\}$.
Denote by $\bar{X}_A$ the shift space of 
the two-sided topological Markov shift $(\bar{X}_A, \bar{\sigma}_A)$
for $A$.
The ordered cohomology group $(\bar{H}^A,\bar{H}^A_+)$
is defined by the quotient group of the ordered abelian group
$C(\bar{X}_A,{\mathbb{Z}})$ of all ${\mathbb{Z}}$-valued continuous functions 
on $\bar{X}_A$ quoted by the subgroup
$\{ \xi - \xi \circ \bar{\sigma}_A \mid \xi \in C(\bar{X}_A,{\mathbb{Z}}) \}$.
The positive cone $\bar{H}^A_+$ consists of the classes of nonnegative functions in $C(\bar{X}_A,{\mathbb{Z}})$ 
(cf. \cite{BH}, \cite{Po}).
We similarly define the ordered cohomology group
$(H^A,H^A_+)$
for one-sided topological Markov shift $(X_A,\sigma_A)$.
The latter ordered group  
$(H^A,H^A_+)$ is naturally isomorphic to the former one  
$(\bar{H}^A,\bar{H}^A_+)$ (\cite[Lemma 3.1]{MM}).
In \cite{BH}, Boyle-Handelman have proved that 
the ordered cohomology group
$(\bar{H}^A,\bar{H}^A_+)$ 
is a complete invariant for  flow equivalence  of 
two-sided topological Markov shift
$(\bar{X}_A, \bar{\sigma}_A)$.
Continuous orbit equivalence of one-sided topological Markov shifts 
is regarded as a counterpart for flow equivalence  of 
two-sided topological Markov shifts 
(see \cite[Theorem 2.3]{MM}, \cite[Corollary 3.8]{MM}).
It is closely related to the classifications of 
both the \'etale groupoids associated to the one-sided Markov shifts
and the Cuntz-Krieger algebras
(see \cite{MaPacific}, \cite{MaPAMS}, \cite{MM},
\cite{MatuiPLMS}, \cite{MatuiPre2012}, 
cf. \cite{Hu}, \cite{Hu2}, \cite{Re}, \cite{Ro}). 
By using the above Boyle-Handelman's result,
 it has been proved in \cite{MM} that
continuous orbit equivalence of one-sided topological Markov 
shift
$(X_A, \sigma_A)$ yields flow equivalence of  
their two-sided topological Markov shift
$(\bar{X}_A, \bar{\sigma}_A)$.
By Parry-Sullivan \cite{PS},
this implies that 
the determinant 
$\det(\id-A)$ is invariant under 
continuous orbit equivalence of one-sided topological Markov 
shift
$(X_A, \sigma_A)$.
Let
$G_A$ denote the \'etale groupoid
for $(X_A,\sigma_A)$ 
whose reduced groupoid $C^*$-algebra 
$C^*_{r}(G^A)$ is isomorphic to the Cuntz-Krieger algebra 
 $\mathcal{O}_A$
(see \cite{MM}, \cite{MatuiPLMS}, \cite{MatuiPre2012}).
As a result, it has been shown that the following three assertions 
for one-sided topological Markov shifts
$(X_A, \sigma_A)$ and $(X_B,\sigma_B)$
are equivalent (\cite[Theorem 3.6]{MM}):
\begin{enumerate}
\renewcommand{\labelenumi}{(\roman{enumi})}
\item
 $(X_A, \sigma_A)$ and $(X_B,\sigma_B)$ 
are continuously orbit equivalent.
\item The \'etale groupoids $G_A$ and $G_B$ are isomorphic. 
\item The Cuntz-Krieger algebras $\mathcal{O}_A$ and $\mathcal{O}_B$ 
are isomorphic and $\det(\id-A)=\det(\id-B)$. 
\end{enumerate}
The method in \cite{MM}
by which 
continuous orbit equivalence of one-sided topological Markov 
shifts
yields  flow equivalence of the 
two-sided topological Markov shifts
has been due to a technique of the groupoids
associated with the one-sided topological Markov shifts.

In this paper, 
we will study 
eventually periodic points and cocycle functions of 
continuously orbit equivalent
one-sided topological Markov shifts $(X_A,\sigma_A)$
and $(X_B,\sigma_B)$.
We then directly construct an isomorphism
between their ordered cohomology groups
$(H^A, H^A_+)$ and
$(H^B, H^B_+)$
without using groupoid.
Let $A, B$ be square irreducible matrices with entries in $\{0,1 \}$.
Suppose that
the one-sided topological Markov shifts
$(X_A, \sigma_A)$ and $(X_B,\sigma_B)$ 
are continuously orbit equivalent via a homeomorphism
$h: X_A \longrightarrow X_B$ so that
\begin{align}
\sigma_B^{k_1(x)} (h(\sigma_A(x))) 
& = \sigma_B^{l_1(x)}(h(x))
\quad \text{ for}
\quad x \in X_A, \label{eq:orbiteqx} \\
\sigma_A^{k_2(y)} (h^{-1}(\sigma_B(y))) 
& = \sigma_A^{l_2(y)}(h^{-1}(y))
\quad \text{ for }
\quad y \in X_B \label{eq:orbiteqy}
\end{align}
for some continuous functions 
$k_1,l_1 \in C(X_A, {\mathbb{Z}}), 
 k_2,l_2 \in C(X_B, {\mathbb{Z}}). 
$
We will directly construct a map
$\Psi_h: C(X_B,{\mathbb{Z}}) \longrightarrow 
C(X_A,{\mathbb{Z}})$
which yields an isomorphism
from $H^B$ to $H^A$ as abelian groups.
We call the functions 
$c_1(x) = l_1(x) - k_1(x), x \in X_A
$ 
and
$ 
c_2(y) = l_2(y) - k_2(y), y \in X_B
$ 
the cocycle functions for $h$ and $h^{-1}$
respectively.
We will prove that 
the classes  $[c_1]$ in $H^A$ and  $[c_2]$ in $H^B$  
of 
$c_1$
and
$ c_2$ give rise to positive elements in the ordered groups
$(H^A, H^A_+)$ and $(H^B, H^B_+)$ 
respectively.
By using the positivities of 
$[c_1]$ and $[c_2]$,
we will show that 
$\Psi_h: C(X_B,{\mathbb{Z}}) \longrightarrow 
C(X_A,{\mathbb{Z}})$
induces an isomorphism of the ordered groups 
from $(H^B, H^B_+)$
to $(H^A, H^A_+)$  (Theorem \ref{thm:ordercoho}, cf. \cite{MM}).

 Continuous orbit equivalence relation 
 of one-sided topological Markov shifts 
 preserve their eventually periodic points
(Proposition \ref{prop:eventually}),
 so that the sets of periodic orbits of 
 the associated two-sided topological Markov shifts are preserved.
Hence there are some relation between their zeta functions.
We will show that 
 the dynamical zeta function 
 $\zeta_{[c_1]}(t)$ for the cocycle function $c_1$
 coincides with the zeta function $\zeta_B(t)$
 of the Markov shift $(\bar{X}_B, \bar{\sigma}_B)$ (Theorem \ref{thm:zeta}). 
 Namely 
 \begin{equation*}
 \zeta_{[c_1]}(t) = \zeta_B(t) \quad 
 \text{ and similarly }\quad 
 \zeta_{[c_2]}(t) = \zeta_A(t). 
 \end{equation*}
It is well-known that periodic points of a tansformation 
gives rise to invariant probability measures  
on the space by averaging the point mass.
As the continuous orbit equivalence preserves structure of periodic orbits,
it is reasonable to have a relationship between their 
shift-invariant measures.  
We will show that 
there exists an order isomorphism
between 
the set of $\sigma_A$-invariant regular Borel measures 
on $X_A$ and
the set of $\sigma_B$-invariant regular Borel measures 
on $X_B$.
If in particular,
the class $[c_1]$ (resp. $[c_2]$) of the cocycle function 
$c_1$ (resp. $c_2$) is cohomologus to $1$ in $H^A$ (resp. $H^B$),
 there exists an affine isomorphism
between 
the set of $\sigma_A$-invariant regular Borel probability measures 
on $X_A$ and
the set of $\sigma_B$-invariant regular Borel probability measures 
on $X_B$ (Theorem \ref{thm:measure}).
Hence the set of shift-invariant regular Borel measures on the one-sided 
topological Markov shift is invariant under continuous orbit equivalence.



Throughout the paper,
we denote by ${\mathbb{N}}$
the set of positive integers
and
by $\Zp$
the set of nonnegative integers,
respectively.

\section{Preliminaries}

Let $A=[A(i,j)]_{i,j=1}^N$ 
be an $N\times N$ matrix with entries in $\{0,1\}$,
where $1< N \in {\Bbb N}$.
Throughout the paper, 
we assume that $A$ has no rows or  columns identically equal to  zero. 
We denote by 
$X_A$ the shift space 
$$
X_A = \{ (x_n )_{n \in \Bbb N} \in \{1,\dots,N \}^{\Bbb N}
\mid
A(x_n,x_{n+1}) =1 \text{ for all } n \in {\Bbb N}
\}
$$
of the right one-sided topological Markov shift for $A$.
It is a compact Hausdorff space in natural  product topology
on $\{1,\dots,N\}^{\Bbb N}$.
The shift transformation $\sigma_A$ on $X_A$ defined by 
$\sigma_{A}((x_n)_{n \in \Bbb N})=(x_{n+1} )_{n \in \Bbb N}$
is a continuous surjective map on $X_A$.
The topological dynamical system 
$(X_A, \sigma_A)$ is called the (right) one-sided topological Markov shift for $A$.
We henceforth assume that 
$A$ is irreducible and satisfies condition (I) in the sense of Cuntz--Krieger \cite{CK}.

A word $\mu = \mu_1 \cdots \mu_k$ for $\mu_i \in \{1,\dots,N\}$
is said to be admissible for $X_A$ 
if $\mu$ appears in somewhere in an element $x$ in $X_A$.
The length of $\mu$ is $k$ and denoted by $|\mu|$.
 We denote by 
$B_k(X_A)$ the set of all admissible words of length $k$.
We set 
$B_*(X_A) = \cup_{k=0}^\infty B_k(X_A)$ 
where $B_0(X_A)$ denotes  the empty word $\emptyset$.
For $x = (x_n )_{n \in \Bbb N} \in X_A$ and 
$k,l \in {\mathbb{N}}$ with $k \le l$,
we set 
\begin{align*}
x_{[k,l]} & = x_k x_{k+1}\cdots x_l \in B_{l-k+1}(X_A),\\
x_{[k,\infty)} & = (x_k,x_{k+1},\dots ) \in X_A.
\end{align*}

For $x = (x_n )_{n \in \Bbb N} \in X_A$,
the orbit $orb_{\sigma_A}(x)$ of $x$ under $\sigma_A$ 
is defined by
$$
orb_{\sigma_A}(x) 
= \cup_{k=0}^\infty \cup_{l=0}^\infty 
\sigma_A^{-k}(\sigma_A^l(x)) \subset X_A.
$$
Let $(X_A, \sigma_A)$ and $(X_B,\sigma_B)$ 
be two topological Markov shifts.
If there exists a homeomorphism 
$h:X_A \longrightarrow X_B$ such that 
$h(orb_{\sigma_A}(x)) = orb_{\sigma_B}(h(x))$ for $x \in X_A$,
then 
$(X_A, \sigma_A)$ and $(X_B,\sigma_B)$ 
are said to be topologically orbit equivalent.
In this case, 
one has 
$h(\sigma_A(x)) \in 
\cup_{k=0}^\infty \cup_{l=0}^\infty \sigma_B^{-k}(\sigma_B^{l}(h(x)))$
for $x \in X_A$.
Hence
there exist $k_1,\, l_1:X_A \rightarrow \Zp$
such that
\begin{equation}
\sigma_B^{k_1(x)} (h(\sigma_A(x))) = \sigma_B^{l_1(x)}(h(x))
\quad \text{ for}
\quad x \in X_A. \label{eq:orbiteqx}
\end{equation}
Similarly 
there exist $k_2,\, l_2:X_B \rightarrow \Zp$
such that
\begin{equation}
\sigma_A^{k_2(y)} (h^{-1}(\sigma_B(y))) = \sigma_A^{l_2(y)}(h^{-1}(y))
\quad \text{ for }
\quad y \in X_B. \label{eq:orbiteqy}
\end{equation}
If we may take 
$k_1,\, l_1:X_A \longrightarrow \Zp$
and
$k_2,\, l_2:X_B \longrightarrow \Zp$
as continuous maps,
the topological Markov shifts 
$(X_A, \sigma_A)$ and $(X_B,\sigma_B)$  
are said to be
 {\it continuously orbit equivalent}.
If two one-sided topological Markov shifts are topologically conjugate,
one may take 
$k_1(x) =k_2(y)=0$
and
$l_1(x) =l_2(y)=1$
so that
they are continuously orbit equivalent.
For the two matrices
$A = 
\begin{bmatrix}
1 & 1 \\
1 & 1
\end{bmatrix}
$
and
$B = 
\begin{bmatrix}
1 & 1 \\
1 & 0
\end{bmatrix},
$
the topological Markov shifts   
$(X_A, \sigma_A)$ and $(X_B, \sigma_B)$
are continuously orbit equivalent, 
but not topologically conjugate (see \cite[Section 5]{MaPacific}).




Throughout the paper,
we assume that one-sided topological Markov shifts 
$(X_A, \sigma_A)$ and $(X_B,\sigma_B)$  
are
continuously orbit equivalent.
We fix 
a homeomorphism
$h: X_A \rightarrow X_B$
and  continuous functions
$k_1,l_1:X_A \rightarrow \Zp$
and
$k_2,l_2:X_B \rightarrow \Zp$
satisfying the equalities
\eqref{eq:orbiteqx} and \eqref{eq:orbiteqy}.

\section{Eventually periodic points}

In this section, we will show that the set of eventually periodic points 
is invariant under continuous orbit equivalence.
For $n \in {\mathbb{N}}$, put
\begin{align*}
k_1^n(x) = \sum_{i=0}^{n-1}k_1(\sigma_A^i(x)),\qquad
& l_1^n(x) = \sum_{i=0}^{n-1}l_1(\sigma_A^i(x)), \qquad x \in X_A, \\
k_2^n(y) = \sum_{i=0}^{n-1}k_2(\sigma_B^i(y)),\qquad
&l_2^n(y) = \sum_{i=0}^{n-1}l_2(\sigma_B^i(y)), \qquad y \in X_B. 
\end{align*}
We note that the following identities hold.
\begin{lemma}
\begin{align}
k_1^{n+m}(x) & = k_1^n(x) + k_1^m(\sigma_A^n(x)), \qquad x \in X_A, \\
l_1^{n+m}(x) & = l_1^n(x) + l_1^m(\sigma_A^n(x)), \qquad x \in X_A, \\
k_2^{n+m}(y) & = k_2^n(y) + k_2^m(\sigma_B^n(y)), \qquad y \in X_B, \\
l_2^{n+m}(y) & = l_2^n(y) + l_2^m(\sigma_B^n(y)), \qquad y \in X_B,
\end{align}
and
\begin{align}
\sigma_B^{k_1^n(x)}(h(\sigma_A^n(x)))
& = \sigma_B^{l_1^n(x)}(h(x)), \qquad x \in X_A, \label{eq:norbiteqx} \\
\sigma_A^{k_2^n(y)}(h^{-1}(\sigma_B^n(y))) 
& = \sigma_A^{l_2^n(y)}(h^{-1}(y)),
\qquad y \in X_B. \label{eq:sAk2n}
\end{align}
\end{lemma}
\begin{lemma}\label{lem:taileq}
Keep the above notations.
\begin{enumerate}
\renewcommand{\labelenumi}{(\roman{enumi})}
\item
If 
$x, z \in X_A$ satisfy
$\sigma_A^p (x)=\sigma_A^q (z)
$
for some
$p,q \in \Zp$, 
then we have
\begin{equation*}
\sigma_B^{l_1^p(x) +k_1^q(z)}(h(x))
=
\sigma_B^{k_1^p(x) +l_1^q(z)}(h(z)).
\end{equation*}
\item
If 
$y, w \in X_B$ satisfy
$\sigma_B^r (y)=\sigma_B^s(w)
$
for some
$r,s \in \Zp$, 
then we have
\begin{equation*}
\sigma_A^{l_2^r(y) +k_2^s(w)}(h^{-1}(y))
=\sigma_A^{k_2^r(y) +l_2^s(w)}(h^{-1}(w)).
\end{equation*}
\end{enumerate}
\end{lemma}
\begin{proof}
(i)
Put 
$u = \sigma_A^p (x)=\sigma_A^q(z) \in X_A$.
It follows that by \eqref{eq:norbiteqx}
\begin{equation*}
\sigma_B^{l_1^{p}(x)}(h(x)) =
\sigma_B^{k_1^{p}(x)}(h(\sigma_A^{p}(x))) =
\sigma_B^{k_1^{p}(x)}(h(u)), 
\end{equation*}
and similarly 
$
\sigma_B^{l_1^{q}(z)}(h(z)) =
\sigma_B^{k_1^{q}(z)}(h(u))
$
so that 
\begin{equation*}
\sigma_B^{l_1^p(x) +k_1^q(z)}(h(x)) =
\sigma_B^{k_1^{q}(z) + k_1^{p}(x)}(h(u)) =
\sigma_B^{k_1^p(x) +l_1^q(z)}(h(z)).
\end{equation*}
(ii) is similarly shown to (i).
\end{proof}
A point $x \in X_A$
is said to be eventually periodic if
there exist $p,q \in \Zp$
with $p \ne q$ such that 
$\sigma_A^p(x) = \sigma_A^q(x)$.
The number $|p-q|$ is called the eventual period
of $x$.
The least number of the eventual periods of $x$ 
is called the least eventual period of $x$. 
If in particular
$\sigma_A^p(x) = x$
for some $p\in {\mathbb{N}}$, 
$x \in X_A$
is said to be $p$-periodic .

The two identites (i) and (ii) in the following lemma 
play important r{\^{o}}le in our further discussions.
\begin{lemma}\label{lem:klp}
For $x \in X_A$, $y \in X_B$ and $p \in \Zp$, 
we have 
\begin{enumerate}
\renewcommand{\labelenumi}{(\roman{enumi})}
\item
$
k_2^{l_1^p(x)}(h(x)) + l_2^{k_1^p(x)}(h(\sigma_A^p(x))) + p
=
k_2^{k_1^p(x)}(h(\sigma_A^p(x))) + l_2^{l_1^p(x)}(h(x)).
$
\item
$
k_1^{l_2^p(y)}(h^{-1}(y)) + l_1^{k_2^p(y)}(h^{-1}(\sigma_B^p(y))) + p
=
k_1^{k_2^p(y)}(h^{-1}(\sigma_B^p(y))) + l_1^{l_2^p(y)}(h^{-1}(y)).
$
\end{enumerate}
\end{lemma}
\begin{proof}
(i)
Put $n = l_1^p(x), m= k_1^p(x)$.
By  \eqref{eq:norbiteqx}, 
one has
\begin{equation}
  h^{-1}(\sigma_B^m(h(\sigma_A^p(x)))) 
= h^{-1}(\sigma_B^n(h(x))). \label{eq:hsBk1n}
\end{equation}
By applying 
$\sigma_A^{k_2^m(h(\sigma_A^p(x))) + k_2^n(h(x))}$
to
\eqref{eq:hsBk1n},
one has
\begin{equation}
\sigma_A^{k_2^m(h(\sigma_A^p(x))) + k_2^n(h(x))}(
h^{-1}(\sigma_B^m(h(\sigma_A^p(x)))))
 =
\sigma_A^{k_2^m(h(\sigma_A^p(x))) + k_2^n(h(x))}(
 h^{-1}(\sigma_B^n(h(x)))). \label{eq:sAhsBk1n}
\end{equation}
The left hand side of \eqref{eq:sAhsBk1n}
goes to
\begin{equation*}
\sigma_A^{k_2^n(h(x))}(
\sigma_A^{k_2^m(h(\sigma_A^p(x)))}(
h^{-1}(\sigma_B^m(h(\sigma_A^p(x))))))
 =
\sigma_A^{k_2^n(h(x))+ l_2^m(h(\sigma_A^p(x))) +p}(x).
\end{equation*}
The right hand side of \eqref{eq:sAhsBk1n}
goes to
\begin{equation*}
\sigma_A^{k_2^m(h(\sigma_A^p(x)))}(
\sigma_A^{k_2^n(h(x))}(
 h^{-1}(\sigma_B^n(h(x)))))
=
\sigma_A^{k_2^m(h(\sigma_A^p(x)))+ l_2^n(h(x))}(x).
\end{equation*}
Hence we have
\begin{equation*}
\sigma_A^{k_2^{l_1^p(x)}(h(x))+ l_2^{k_1^p(x)}(h(\sigma_A^p(x))) +p}(x)
=
\sigma_A^{k_2^{k_1^p(x)}(h(\sigma_A^p(x)))+ l_2^{l_1^p(x)}(h(x))}(x).
\end{equation*}
Now suppose that there exists $x \in X_A$ such that 
\begin{equation}
k_2^{l_1^p(x)}(h(x))+ l_2^{k_1^p(x)}(h(\sigma_A^p(x))) +p
\ne
k_2^{k_1^p(x)}(h(\sigma_A^p(x)))+ l_2^{l_1^p(x)}(h(x)) \label{eq:ne}
\end{equation}
so that
$x$ is an eventually periodic point.
Since
the functions $k_1, l_1, k_2, l_2$
are all continuous, 
\eqref{eq:ne} hold for all elements of a neighborhood of $x$.
Hence there exists an open set of $X_A$ whose  elements
are all eventually periodic points.
It is a contradiction to the fact that 
 the set of all non eventually periodic points is dense in $X_A$.
Therefore the identity  
\begin{equation*}
k_2^{l_1^p(x)}(h(x))+ l_2^{k_1^p(x)}(h(\sigma_A^p(x))) +p
=
k_2^{k_1^p(x)}(h(\sigma_A^p(x)))+ l_2^{l_1^p(x)}(h(x)) 
\end{equation*}
holds for all $x \in X_A$.

(ii) is similarly shown to (i).
\end{proof}
\begin{lemma}\label{lem:periodic}
Let
$x$ be a periodic point in $X_A$.
Then $h(x)$ is an eventually periodic point in $X_B$.
\end{lemma}
\begin{proof}
Assume that 
$\sigma^p(x) = x$ for some $p \in {\mathbb{N}}$.
By the above lemma (i),
we have
\begin{equation*}
k_2^{l_1^p(x)}(h(x)) + l_2^{k_1^p(x)}(h(x)) + p
=
k_2^{k_1^p(x)}(h(x)) + l_2^{l_1^p(x)}(h(x))
\end{equation*}
so that $l_1^p(x) \ne k_1^p(x)$.
By the identity \eqref{eq:norbiteqx}
with $\sigma_A^p(x) = x$,
one has
\begin{equation}
\sigma_B^{k_1^p(x)}(h(x)) = \sigma_B^{l_1^p(x)}(h(x))
\label{eq:3.10}
\end{equation}
which implies that
$h(x)$ is an eventually periodic point in $X_B$.
\end{proof}

\begin{proposition}\label{prop:eventually}
Suppose that one-sided topological Markov shifts 
$(X_A, \sigma_A)$ and $(X_B,\sigma_B)$  
are continuously orbit equivalent via
a homeomorphism
$h: X_A \rightarrow X_B$
satisfying the equalities
\eqref{eq:orbiteqx} and \eqref{eq:orbiteqy}.
Let
$x$ be an eventually periodic point in $X_A$.
Then $h(x)$ is an eventually periodic point in $X_B$.
Therefore the set of eventually periodic points of 
a one-sided topological Markov shift 
is invariant under continuous orbit equivalence.
\end{proposition}
\begin{proof}
Let
$x$ be an eventually periodic point in $X_A$
such that
$\sigma_A^{p+q}(x) = \sigma_A^p(x)$
for some $p \in \Zp, q \in {\mathbb{N}}$.
Put $\tilde{x} = \sigma_A^p(x)$. 
 By the preceding lemma,
 $h(\tilde{x})$ is an eventually periodic point in $X_B$.
Take $p_1,p_2 \in \Zp$ with $p_1\ne p_2$
such that 
$\sigma_B^{p_1}(h(\tilde{x})) = \sigma_B^{p_2}(h(\tilde{x}))$.
By Lemma \ref{lem:taileq}, 
 there exist
$q_1,q_2 \in \Zp$ 
such that 
$\sigma_B^{q_1}(h(x)) = \sigma_B^{q_2}(h(\tilde{x}))$,
so that we have
\begin{equation*}
  \sigma_B^{p_1+q_1}(h(x)) 
= \sigma_B^{p_1}(\sigma_B^{q_2}(h(\tilde{x})))
= \sigma_B^{p_2}(\sigma_B^{q_2}(h(\tilde{x})))
= \sigma_B^{p_2+q_1}(h(x)). 
\end{equation*}
Since
$p_1 + q_1 \ne p_2 + q_1$,
$h(x)$ is an eventually periodic point in $X_B$.
\end{proof}

\section{Construction of an isomorphism 
from $C(X_B,{\mathbb{Z}})$ to $C(X_A,{\mathbb{Z}})$}
We are assuming that one-sided topological Markov shifts 
$(X_A, \sigma_A)$ and $(X_B,\sigma_B)$  
are
continuously orbit equivalent through 
a homeomorphism
$h$ from $X_A$ to $X_B$
with  continuous functions
$k_1,l_1:X_A \longrightarrow \Zp$
and
$k_2,l_2:X_B \longrightarrow \Zp$
satisfying 
\eqref{eq:orbiteqx} and \eqref{eq:orbiteqy}.
Let us denote by 
$C(X_A,{\mathbb{Z}})$ (resp. $C(X_B,{\mathbb{Z}})$)
the abelian group of all integer valued continuous functions
on $X_A$ (resp. $X_B$).
In this section, we will directly construct 
an isomorphism from $C(X_B,{\mathbb{Z}})$ to $C(X_A,{\mathbb{Z}})$
compatible to the shifts.
 \begin{lemma}\label{lem:cocycle}
The function $c_1(x) = l_1(x) - k_1(x)$ for $x \in X_A$
does not depend on the choice of $k_1, l_1 \in C(X_A,\Z)$
as long as  satisfing   
\eqref{eq:orbiteqx}.
\end{lemma}
\begin{proof}
Let 
$k'_1, l'_1 \in C(X_A,\Z)$
be another functions for $k_1, l_1$ satisfying
\begin{equation}
\sigma_B^{k'_1(x)} (h(\sigma_A(x))) = \sigma_B^{l'_1(x)}(h(x))
\quad \text{ for}
\quad x \in X_A. \label{eq:primeorbiteqx}
\end{equation}
Since $k_1, k'_1$ are both continuous,
there exists $K \in \N$ such that 
$k_1(x), k'_1(x) \le K$ for all $x \in X_A$. 
Put
$c'_1(x) = l'_1(x) - k'_1(x)$ for $x \in X_A$
so that
\begin{equation*}
 \sigma_B^{c_1(x)+K}(h(x)) = \sigma_B^{c'_1(x)+K}(h(x))
 \quad \text{ for all } x \in X_A. 
\end{equation*}
Suppose that 
$c_1(x_0) \ne c'_1(x_0)$ for some $x_0 \in X_A$.
There exists a clopen neighborhood $U$ of $x_0$
such that 
$
c_1(x) \ne c'_1(x) 
$
for all 
$ x \in U.$
As  
$
c_1(x) +K \ne c'_1(x) +K
$
for all $x \in U$,
the points $h(x)$ for all $x \in U$ 
are eventually periodic points,
which is a contradiction to
the fact that the set of non eventualy periodic points is 
dense in $X_B$.
Therefore we conclude that
$c_1(x) = c'_1(x)$ for all $x \in X_A$.
\end{proof}

For $f \in C(X_B,\Z )$, define
\begin{equation}
\Psi_{h}(f)(x)
= \sum_{i=0}^{l_1(x)-1} f(\sigma_B^i(h(x))) 
- \sum_{j=0}^{k_1(x)-1} f(\sigma_B^j(h(\sigma_A(x)))),
 \quad 
 x \in X_A. \label{eq:Psihfx}
\end{equation}
It is easy to see that 
$\Psi_h(f) \in C(X_A,\Z)$.
Thus $\Psi_h: C(X_B,\Z) \longrightarrow C(X_A,\Z)$
gives rise to a homomorphism of abelian groups.

\begin{lemma}
$\Psi_h: C(X_B,\Z) \longrightarrow C(X_A,\Z)$
does not depend on the choice of the functions 
$k_1, l_1$
as long as  satisfing   
\eqref{eq:orbiteqx}.
\end{lemma}
\begin{proof}
Let 
$k'_1, l'_1 \in C(X_A,\Z)$
be another functions for $k_1, l_1$ 
satisfying
\eqref{eq:primeorbiteqx}.
We fix an arbitrary $x \in X_A$.
By Lemma \ref{lem:cocycle},
we see that
$l_1(x) - k_1(x) = l'_1(x) - k'_1(x)$.
Assume that
$l'_1(x) < l_1(x)$
so that
$k_1(x) - k'_1(x) = l_1(x) - l'_1(x)>0$.
We put
\begin{equation*}
\Psi'_{h}(f)(x)
= \sum_{i=0}^{l'_1(x)-1} f(\sigma_B^i(h(x))) 
- \sum_{j=0}^{k'_1(x)-1} f(\sigma_B^j(h(\sigma_A(x)))), 
\end{equation*}
so that 
\begin{equation*}
\Psi_{h}(f)(x) - \Psi'_{h}(f)(x)
= \sum_{i=l'_1(x)}^{l_1(x)-1} f(\sigma_B^i(h(x)))
- \sum_{j=k'_1(x)}^{k_1(x)-1} f(\sigma_B^j(h(\sigma_A(x)))). 
\end{equation*}
As
$\sigma_B^{l'_1(x)+j}(h(x)) = \sigma_B^{k'_1(x)+j}(h(\sigma_A(x)))
$
for $j=0,1,\dots, l_1(x)-l'_1(x)-1 (=k_1(x)-k'_1(x)-1)$,
we see that
$\Psi_{h}(f)(x) - \Psi'_{h}(f)(x)
=0$.
\end{proof}

The equalities in the following lemma are 
basic in our further discussions.
\begin{lemma} \label{lem:Am}
For $f \in C(X_B,{\mathbb{Z}})$,
$x \in X_A$ and $m=1,2,\dots,$ 
the following equalities hold:
\begin{align*}
& \sum_{i=0}^{m-1} \{
\sum_{i'=0}^{l_1(\sigma_A^i(x))-1} f(\sigma_B^{i'}(h(\sigma_A^i(x))))
-
\sum_{j'=0}^{k_1(\sigma_A^i(x))-1} f(\sigma_B^{j'}(h(\sigma_A^{i+1}(x)))) \}\\
=
&
\sum_{i'=0}^{l_1^m (x)-1} f(\sigma_B^{i'}(h(x)))
-
\sum_{j'=0}^{k_1^m (x)-1} f(\sigma_B^{j'}(h(\sigma_A^{m}(x)))).
\end{align*}
\end{lemma}
\begin{proof}
For $m=1$, the left hand side of the desired equality is
\begin{equation*}
\sum_{i'=0}^{l_1(x)-1}f(\sigma_B^{i'}(h(x)))
-
\sum_{j'=0}^{k_1(x)-1} f(\sigma_B^{j'}(h(\sigma_A(x)))).
\end{equation*}
which is equal to the right hand side of the desired equality. 

We assume that the desired formula holds for some $m$.
It then follows that 
\begin{align*}
& \sum_{i=0}^{m} \{
\sum_{i'=0}^{l_1 (\sigma_A^i(x))-1} 
f(\sigma_B^{i'}(h(\sigma_A^i(x))))
-\sum_{j'=0}^{k_1 (\sigma_A^i(x))-1} 
f(\sigma_B^{j'}(h(\sigma_A^{i+1}(x)))) \}\\
=
& \sum_{i=0}^{m-1} \{
\sum_{i'=0}^{l_1 (\sigma_A^i(x))-1} 
f(\sigma_B^{i'}(h(\sigma_A^i(x))))
-
\sum_{j'=0}^{k_1 (\sigma_A^i(x))-1} 
f(\sigma_B^{j'}(h(\sigma_A^{i+1}(x)))) \}\\
+
&
\sum_{i'=0}^{l_1 (\sigma_A^m(x))-1} 
f(\sigma_B^{i'}(h(\sigma_A^m(x))))
-
\sum_{j'=0}^{k_1 (\sigma_A^m(x))-1} 
f(\sigma_B^{j'}(h(\sigma_A^{m+1}(x)))) \\
=
&  
\sum_{i'=0}^{l_1^m (x)-1} 
f(\sigma_B^{i'}(h(x)))
-
\sum_{j'=0}^{k_1^m (x)-1} 
f(\sigma_B^{j'}(h(\sigma_A^{m}(x))))  \\
+
&
\sum_{i'=0}^{l_1 (\sigma_A^m(x))-1} 
f(\sigma_B^{i'}(h(\sigma_A^m(x))))
-
\sum_{j'=0}^{k_1 (\sigma_A^m(x))-1} 
f(\sigma_B^{j'}(h(\sigma_A^{m+1}(x)))) \\
=
&
\sum_{i'=0}^{l_1^m (x)-1} 
f(\sigma_B^{i'}(h(x))) \\
-
&
\{
\sum_{j'=0}^{l_1 (\sigma_A^m(x)) + k_1^m(x)-1} 
f(\sigma_B^{j'}(h(\sigma_A^{m} (x)))) 
-
\sum_{j'=k_1^m(x)}^{l_1 (\sigma_A^m(x)) +k_1^m(x)-1} 
f(\sigma_B^{j'}(h(\sigma_A^{m}(x)))) \} \\
+
&
\{
\sum_{i'=0}^{l_1 (\sigma_A^m(x)) + k_1^m(x)-1} 
f(\sigma_B^{i'}(h(\sigma_A^m(x)))) 
-
\sum_{i'=l_1 (\sigma_A^m(x))}^{l_1 (\sigma_A^m(x)) + k_1^m(x)-1} 
f(\sigma_B^{i'}
(h(\sigma_A^m(x)))) \}  
\\
-
&
\sum_{j'=0}^{k_1 (\sigma_A^m(x))-1} 
f(\sigma_B^{j'}(h(\sigma_A^{m+1}(x)))) 
\end{align*}
The second summand of the first $\{ \, \cdot \, \}$ above  goes to
\begin{align*}
\sum_{j'=0}^{l_1 (\sigma_A^m(x)) -1} 
f(\sigma_B^{j'}(\sigma_B^{k_1^m(x)}(h(\sigma_A^{m}(x))))) 
=
&
\sum_{j'=0}^{l_1 (\sigma_A^m(x)) -1} 
f(\sigma_B^{j'}(\sigma_B^{l_1^m(x)}(h(x)))) \\
=
&
\sum_{i'=l_1^m(x)}^{l_1^{m+1}(x) -1} 
f(\sigma_B^{i'}(h(x))).
\end{align*}
The second summand of the second $\{ \, \cdot \, \}$ above goes to
\begin{align*}
 \sum_{i'=0}^{ k_1^m(x)-1} 
f(\sigma_B^{i'}(\sigma_B^{l_1 (\sigma_A^m(x))}
(h(\sigma_A^m(x))))) 
=
& \sum_{i'=0}^{ k_1^m(x)-1} 
f(\sigma_B^{i'}(\sigma_B^{k_1 (\sigma_A^m(x))}
(h(\sigma_A^{m+1}(x))))) \\
=
& \sum_{j'=k_1(\sigma_A^m(x))}^{ k_1^{m+1}(x)-1} 
f(\sigma_B^{j'}(h(\sigma_A^{m+1}(x)))).
\end{align*}
Hence we have
\begin{align*}
& \sum_{i=0}^{m} \{
\sum_{i'=0}^{l_1 (\sigma_A^i(x))-1} 
f(\sigma_B^{i'}(h(\sigma_A^i(x))))
-
\sum_{j'=0}^{k_1 (\sigma_A^i(x))-1} 
f(\sigma_B^{j'}(h(\sigma_A^{i+1}(x)))) \}\\
=
&
\sum_{i'=0}^{l_1^m (x)-1} 
f(\sigma_B^{i'}(h(x)))
+
\sum_{i'=l_1^m(x)}^{l_1^{m+1}(x) -1} 
f(\sigma_B^{i'}(h(x))) \\
-
& 
 \sum_{j'=k_1(\sigma_A^m(x))}^{ k_1^{m+1}(x)-1} 
f(\sigma_B^{j'}(h(\sigma_A^{m+1}(x))))
-
\sum_{j'=0}^{k_1 (\sigma_A^m(x))-1} 
f(\sigma_B^{j'}(h(\sigma_A^{m+1}(x)))) \\
=
&
\sum_{i'=0}^{l_1^{m+1} (x)-1} f(\sigma_B^{i'}(h(x)))
- 
 \sum_{j'=0}^{ k_1^{m+1}(x)-1} 
f(\sigma_B^{j'}(h(\sigma_A^{m+1}(x))))
\end{align*}
which shows that the desired equality for $m+1$ holds.
\end{proof}

For $g \in C(X_A,\Z )$,
let us define 
$\Psi_{h^{-1}}(g) \in C(X_B,\Z)$
by substituting $h^{-1}$ for $h$ in \eqref{eq:Psihfx}
as follows:
\begin{equation}
\Psi_{h^{-1}}(g)(y)
= \sum_{i=0}^{l_2(y)-1} g(\sigma_A^i(h^{-1}(y))) 
- \sum_{j=0}^{k_2(y)-1} g(\sigma_A^j(h^{-1}(\sigma_B(y)))),
 \quad 
 y \in X_B. \label{eq:Psiinverseh}
\end{equation}
In Lemma \ref{lem:Am},
by substituting 
$h(x)$(resp. $h(\sigma_A(x))$) for $x$,
$\sigma_B$ for $\sigma_A$,
$h^{-1}$ for $h$,
$\sigma_A$ for $\sigma_B$,
$l_2$ for $l_1$,
$k_2$ for $k_1$,
and
$m=l_1(x)$ (resp. $m=k_1(x)$),
respectively,
we have
the following lemma (i)(resp. (ii)).
\begin{lemma}\label{lem:Psi}
For $g \in C(X_A,{\mathbb{Z}})$ and $x \in X_A$, we have 
\begin{enumerate}
\renewcommand{\labelenumi}{(\roman{enumi})}
\item
\begin{align*}
& \sum_{i=0}^{l_1(x)-1} 
\Psi_{h^{-1}}(g)(\sigma_B^i(h(x)) )\\
=
&
\sum_{i'=0}^{l_2^{l_1(x)} (h(x))-1} 
g(\sigma_A^{i'}(x))
-
\sum_{j'=0}^{k_2^{l_1(x)}(h(x))-1}  
g(\sigma_A^{j'}(h^{-1}(\sigma_B^{l_1(x)}(h(x))))).
\end{align*}
\item
\begin{align*}
& \sum_{j=0}^{k_1(x)-1} 
\Psi_{h^{-1}}(g)(\sigma_B^j(h(\sigma_A(x)))) \\
=
&
\sum_{i'=0}^{l_2^{k_1(x)} (h(\sigma_A(x)))-1} 
g(\sigma_A^{i'+1}(x))
-
\sum_{j'=0}^{k_2^{k_1(x)}(h(\sigma_A(x)))-1}  
g(\sigma_A^{j'}(h^{-1}(\sigma_B^{k_1(x)}(h(\sigma_A(x)))))).
\end{align*}
\end{enumerate}
\end{lemma}


We will prove the following proposition.
\begin{proposition}\label{prop:compo} 
$\Psi_h \circ \Psi_{h^{-1}} = \id_{C(X_A,{\mathbb{Z}})}$
and similarly
$\Psi_{h^{-1}}\circ \Psi_h = \id_{C(X_B,{\mathbb{Z}})}$.
\end{proposition}
\begin{proof}
We put
\begin{align*}
k_3(x) &=k_2^{l_1(x)}(h(x)) + l_2^{k_1(x)}(h(\sigma_A(x))), \qquad x \in X_A, \\
l_3(x) &=l_2^{l_1(x)}(h(x)) + k_2^{k_1(x)}(h(\sigma_A(x))), \qquad x \in X_A.
\end{align*}
By Lemma \ref{lem:klp} we have 
$k_3(x) +1 = l_3(x)$
so that the identity
for $g \in C(X_A,\Z)$ 
\begin{equation}
g(x)
= \sum_{i=0}^{l_3(x)-1} g(\sigma_A^i(x)) 
- \sum_{j=0}^{k_3(x)-1} g(\sigma_A^{j+1}(x)),
 \quad 
 x \in X_A \label{eq:Psighfx}
\end{equation}
holds.
By the above lemma,
we have 
\begin{align*}
  &\Psi_{h}(\Psi_{h^{-1}}(g))(x) \\
=&  \sum_{i=0}^{l_1(x)-1} \Psi_{h^{-1}}(g)(\sigma_B^i(h(x)) )
- \sum_{j=0}^{k_1(x)-1} \Psi_{h^{-1}}(g)(\sigma_B^j(h(\sigma_A(x))) ) \\
=
&
\{ 
\sum_{i'=0}^{l_2^{l_1(x)} (h(x))-1} g(\sigma_A^{i'}(x))
-
\sum_{j'=0}^{k_2^{l_1(x)}(h(x))-1}  
g(\sigma_A^{j'}(h^{-1}(\sigma_B^{l_1(x)}(h(x)))))
\} 
\\
-&
\{
\sum_{i'=0}^{l_2^{k_1(x)} (h(\sigma_A(x)))-1} 
g(\sigma_A^{i'+1}(x))
-
\sum_{j'=0}^{k_2^{k_1(x)}(h(\sigma_A(x)))-1}  
g(\sigma_A^{j'}(h^{-1}(\sigma_B^{k_1(x)}(h(\sigma_A(x))))) \} 
\end{align*}
The first $\{ \, \cdot \, \}$ above  goes to
\begin{align}
& \sum_{i'=0}^{l_3(x)-1} 
g(\sigma_A^{i'}(x)) 
\label{eq:circ1A} \\
-
& \{ \sum_{i'=l_2^{l_1(x)}(h(x))}^{l_3(x)-1} 
g(\sigma_A^{i'}(x)) 
+
 \sum_{j'=0}^{k_2^{l_1(x)}(h(x))-1}  
g(\sigma_A^{j'}(h^{-1}(\sigma_B^{l_1(x)}(h(x))))) \}.
 \label{eq:circ1C}
\end{align}
The second $\{ \, \cdot \, \}$ above goes to
\begin{align}
& \sum_{i'=0}^{k_3(x)-1} 
g(\sigma_A^{i'+1}(x)) 
\label{eq:circ2A} \\
-
& \{
\sum_{i'=l_2^{k_1(x)}(h(\sigma_A(x)))}^{k_3(x)-1} 
g(\sigma_A^{i'+1}(x)) 
+
\sum_{j'=0}^{k_2^{k_1(x)}(h(\sigma_A(x)))-1}  
g(\sigma_A^{j'}(h^{-1}(\sigma_B^{k_1(x)}(h(\sigma_A(x)))))) \}.
\label{eq:circ2C}
\end{align}
Hence we have
\begin{equation*}
\Psi_{h}(\Psi_{h^{-1}}(g))(x)
= \{ \eqref{eq:circ1A} +\eqref{eq:circ1C}\}
 - \{ \eqref{eq:circ2A} +\eqref{eq:circ2C}\} 
\end{equation*}
Since
$
\sigma_A^{l_2^{l_1(x)}(h(x))}(x) 
=\sigma_A^{k_2^{l_1(x)}(h(x))}(h^{-1}(\sigma_B^{l_1(x)}(h(x)))),
$
we  have
\begin{align*}
& - \eqref{eq:circ1C} \\
=
& \sum_{j'=0}^{k_2^{k_1(x)}(h(\sigma_A(x)))-1} 
g(\sigma_A^{j'}(\sigma_A^{l_2^{l_1(x)}(h(x))}(x))) 
+
 \sum_{j'=0}^{k_2^{l_1(x)}(h(x))-1}  
g(\sigma_A^{j'}(h^{-1}(\sigma_B^{l_1(x)}(h(x))))) \\
=
&
 \sum_{j'=0}^{k_2^{l_1(x)}(h(x))+ k_2^{k_1(x)}(h(\sigma_A(x)))-1}  
g(\sigma_A^{j'}(h^{-1}(\sigma_B^{l_1(x)}(h(x))))). 
\end{align*}
Since
$
\sigma_A^{l_2^{k_1(x)}(h(\sigma_A(x)))}(\sigma_A(x)) 
=
\sigma_A^{k_2^{k_1(x)}(h(\sigma_A(x)))}(
h^{-1}(\sigma_B^{k_1(x)}(h(\sigma_A(x)))),
$
we  have
\begin{align*}
& - \eqref{eq:circ2C} \\
=
& \sum_{j'=0}^{ k_2^{l_1(x)}(h(x))-1} 
g(\sigma_A^{j'+ l_2^{k_1(x)}(h(\sigma_A(x)))}(\sigma_A(x)))) 
+
\sum_{j'=0}^{k_2^{k_1(x)}(h(\sigma_A(x)))-1}  
g(\sigma_A^{j'}(h^{-1}(\sigma_B^{k_1(x)}(h(\sigma_A(x))))) \\
=
&
\sum_{j'=0}^{k_2^{k_1(x)}(h(\sigma_A(x))) +k_2^{l_1(x)}(h(x))-1}  
g(\sigma_A^{j'}(h^{-1}(\sigma_B^{k_1(x)}(h(\sigma_A(x))))) 
\end{align*}
We thus have
$
\eqref{eq:circ1C} 
=
\eqref{eq:circ2C} 
$
by
\eqref{eq:orbiteqx}
and
$
\eqref{eq:circ1A}-\eqref{eq:circ2A}
= g(x)
$
by
\eqref{eq:Psighfx}
so that
\begin{equation*}
\Psi_{h}(\Psi_{h^{-1}}(g))(x)
= g(x).
\end{equation*}
Similarly we have
$
\Psi_{h}\circ \Psi_{h^{-1}} = \id_{C(X_A,{\mathbb{Z}})}.
$
\end{proof}

\begin{lemma}\label{lem:cobdy} \hspace{6cm}
\begin{enumerate}
\renewcommand{\labelenumi}{(\roman{enumi})}
\item
$
 \Psi_h(f- f \circ \sigma_B) 
 = f \circ h - f \circ h \circ \sigma_A,
\qquad
f \in C(X_B,\Z). 
$
\item
$ \Psi_{h^{-1}}(g- g \circ \sigma_A) 
 = g \circ h^{-1} - g \circ h^{-1} \circ \sigma_B,
\qquad
g \in C(X_A,\Z). 
$
\end{enumerate}
\end{lemma}
\begin{proof}
(i)
For $f \in C(X_B,\Z)$ and $x \in X_A$,
we have
\begin{align*}
 & \Psi_h(f- f \circ \sigma_B)(x) \\
=& \sum_{i=0}^{l_1(x)-1} (f- f \circ \sigma_B)(\sigma_B^i(h(x))) 
- \sum_{j=0}^{k_1(x)-1} (f- f \circ \sigma_B)(\sigma_B^j(h(\sigma_A(x)))) \\
= &f(h(x)) - f(\sigma_B^{l_1(x)}(h(x))) 
-  f(h(\sigma_A(x))) + f(\sigma_B^{k_1(x)}(h(\sigma_A(x))))\\ 
= & f(h(x)) - f(h(\sigma_A(x))).
\end{align*}
(ii) is similarly shown.
\end{proof}


\section{Ordered cohomology groups}
Let us denote by
\begin{equation*}
\bar{X}_A = \{ (x_n )_{n \in {\mathbb{Z}}} 
\in \{1,\dots,N \}^{\mathbb{Z}}
\mid
A(x_n,x_{n+1}) =1 \text{ for all } n \in {\mathbb{Z}}
\}
\end{equation*}
the shift space of the two-sided topological Markov shift 
$(\bar{X}_A, \bar{\sigma}_A)$ for $A$
with shift transformation
$\bar{\sigma}_A$ on $\bar{X}_A$ 
defined by 
$\bar{\sigma}_{A}((x_n)_{n \in {\mathbb{Z}}})
=(x_{n+1} )_{n \in {\mathbb{Z}}}$
which is a homeomorphism on $\bar{X}_A$.
Set 
\begin{equation*}
\bar{H}^A
=C(\bar{X}_A,{\mathbb{Z}})/\{\xi-\xi\circ\bar\sigma_A\mid\xi
\in C(\bar{X}_A,{\mathbb{Z}})\}. 
\end{equation*}
The equivalence class of a function $\xi\in C(\bar{X}_A, {\mathbb{Z}})$ in 
$\bar{H}^A$ 
is written $[\xi]$. 
We define the positive cone $\bar{H}^A_+$ by 
\begin{equation*}
\bar{H}^A_+=\{[\xi]\in\bar{H}^A
\mid\xi(x)\geq0\quad\forall x\in\bar{X}_A\}. 
\end{equation*}
The pair $(\bar{H}^A,\bar{H}^A_+)$ is called 
the ordered cohomology group of $(\bar{X}_A,\bar\sigma_A)$ 
(see \cite[Section 1.3]{BH}, \cite{Po}). 
M. Boyle and D. Handelman
have proved that 
$(\bar{X}_A,\bar\sigma_A)$ and $(\bar{X}_B,\bar\sigma_B)$ are 
flow equivalent if and only if
the ordered cohomology groups 
$(\bar{H}^A,\bar{H}^A_+)$ and $(\bar{H}^B,\bar{H}^B_+)$ are isomorphic
(\cite[Theorem 1.12]{BH}).

In the same way as above, 
the ordered group
$(H^A,H^A_+)$ 
for the one-sided topological Markov shift $(X_A,\sigma_A)$
has been introduced in \cite{MM} by setting
\begin{equation*}
H^A=C(X_A,{\mathbb{Z}})/\{\xi-\xi\circ\sigma_A\mid\xi\in C(X_A,{\mathbb{Z}})\}
\end{equation*}
and 
\begin{equation*}
H^A_+=\{[\xi]\in H^A\mid\xi(x)\geq0\quad\forall x\in X_A\}. 
\end{equation*}
It has been proved that the ordered groups 
$(\bar H^A,\bar H^A_+)$ and $(H^A,H^A_+)$ are actually isomorphic 
(\cite[Lemma 3.1]{MM}).

By Proposition \ref{prop:compo}
and Lemma \ref{lem:cobdy},
we see
\begin{proposition}
Let $h:X_A \longrightarrow X_B$
be a homeomorphism which gives rise to a continuous orbit equivalence
between $(X_A,\sigma_A)$ and $(X_B,\sigma_B)$.
Then 
$\Psi_h : C(X_B,\Z) \longrightarrow C(X_A,\Z)$
induces an isomomorphism
$\bar{\Psi}_h: H^B \longrightarrow H^A$
of abelian groups in a natural way.
\end{proposition}
In \cite{MM}, it has been proved that 
$(H^A,H^A_+)$ is isomorphic to
$(H^B,H^B_+)$
as ordered groups by groupoid technique.
In this section, we will prove that the 
above isomorphism
$\Psi_h$ preserves their positive cone, that is
$\bar{\Psi}_h(H^B_+) =H^A_+$
without groupoid technique
so that 
$\bar{\Psi}_h$ induces an isomorphism from
$(H^B,H^B_+)$
to $(H^A,H^A_+)$
as ordered groups.

A subset $S \subset X_A$ is said to be $\sigma_A$-invariant
if $\sigma_A(S) = S$.
We similarly say 
$\bar{S} \subset \bar{X}_A$ to be $\bar{\sigma}_A$-invariant
if $\bar{\sigma}_A(\bar{S}) = \bar{S}$.
We note that 
a finite subset $S \subset X_A$ 
is $\sigma_A$-invariant if and only if 
there exists a finite family of periodic points
$x(i), i=1,\dots,m$ such that 
$x(i)$ is $p_i$-periodic for some $p_i \in {\mathbb{N}}$
and
$$
S = \{ \sigma_A^j(x(i)) \in X_A \mid j=0,1,\dots,p_i-1, \, i=1,\dots,m\}.
$$ 
\begin{lemma}[{\cite[Lemma 3.2]{MM}}] \label{lem:6.2}
For $f \in C(X_A,\mathbb{Z})$, we have
$[f]$ belongs to $H^A_+$ if and only if
for every eventually periodic point
$x$ with $\sigma_A^r(x) = \sigma_A^s(x)$ and $r-s >0$,
the value 
\begin{equation}
\omega_f^{r,s}(x) = 
\sum_{i=0}^{r-1} f(\sigma_A^i(x)) -\sum_{j=0}^{s-1} f(\sigma_A^j(x))
\end{equation}
satisfies
$\omega_f^{r,s}(x) \ge 0$.
If in particular $[f]$ is an order unit of $(H^A,H^A_+)$
if and only if $\omega_f^{r,s}(x) >0$.
\end{lemma}
\begin{proof}
For $f \in C(X_A,\mathbb{Z})$, 
it has been shown in \cite[Lemma 3.2]{MM}
that 
$[f]$ belongs to $H^A_+$ if and only if
$\sum_{x \in O}f(x) \ge 0$
for every finite $\sigma_A$-invariant set $O$ of $X_A$.
Let
$x$ be an eventually periodic point
such that $\sigma_A^r(x) = \sigma_A^s(x)$ and $r-s >0$.
Let $p$ be the least period of $\sigma_A^s(x)$
so that $r-s = np$ for some $n \in {\mathbb{N}}$.
It then follows that 
\begin{equation*}
\omega_f^{r,s}(x) 
 = \sum_{i=s}^{r-1} f(\sigma_A^i(x)) 
= n\{ f(\sigma_A^s(x)) + f(\sigma_A^{s+1}(x)) + \cdots+
f(\sigma_A^{s+p-1}(x)) \}.
\end{equation*} 
Since the set 
$O =\{ \sigma_A^s(x), \sigma_A^{s+1}(x), \dots, \sigma_A^{s+p-1}(x) \}$
is a finite $\sigma_A$-invariant set of $X_A$,
one sees that 
$[f] \in H^A_+$ if and only if 
$\omega_f^{r,s}(x) \ge 0$ by \cite[Lemma 3.2]{MM}.
We know that, by \cite[Proposition 3.13]{BH}, 
the class $[f]$ is an order unit 
of $(H^A,H^A_+)$
if and only if $\omega_f^{r,s}(x) >0$.
\end{proof}

\begin{lemma}\label{lem:6.3}
For $x\in X_A$ with
$\sigma_A^r(x) = \sigma_A^s(x)$ and
$r-s = q \in {\mathbb{N}}$,
put
$z = \sigma_B^{l_1^s(x) +k_1^s(x)}(h(x)) \in X_B$
and
$r' = l_1^{q}(\sigma_A^s(x)), s' = k_1^{q}(\sigma_A^s(x))$.
Then we have
\begin{gather}
\sigma_B^{r'}(z) 
= 
\sigma_B^{s'}(z), 
\qquad
r' \ne s', 
\label{eq:l1np}\\
\omega_{\Psi_h(f)}^{r,s}(x)
 = 
\omega_f^{r',s'}(z) \quad \text{ for } \quad f \in C(X_B,\Z). 
\label{eq:oPhfx}
\end{gather}
\end{lemma}
\begin{proof}
As $l_1^r(x) = l_1^s(x) + r'$
and
 $k_1^r(x) = k_1^s(x) + s'$,
 we have
\begin{align*}
\sigma_B^{r'}(z) 
=
& \sigma_B^{k_1^s(x) }(\sigma_B^{l_1^r(x)}(h(x))) \\
=
& \sigma_B^{k_1^s(x) }(\sigma_B^{k_1^r(x)}(h(\sigma_A^r(x)))) \\
=
& \sigma_B^{k_1^r(x) }(\sigma_B^{k_1^s(x)}(h(\sigma_A^s(x)))) \\
=& \sigma_B^{k_1^r(x) }(\sigma_B^{l_1^s(x)}(h(x)))\\
=&
\sigma_B^{s'}(z).
\end{align*}
The identity (i) of Lemma \ref{lem:klp} implies that 
\begin{equation*}
 k_2^{r'}(h(\sigma_A^s(x))) 
+ l_2^{s'}(h(\sigma_A^q(\sigma_A^s(x)))) + q 
=
 k_2^{s'}(h(\sigma_A^q(\sigma_A^s(x)))) 
+ l_2^{r'}(h(\sigma_A^s(x))).
\end{equation*}
As 
$\sigma_A^q(\sigma_A^s(x))
= \sigma_A^r(x) = \sigma_A^s(x)$
and
$q \ne 0$,
we have
$
{r'} \ne {s'}.
$

For $f \in C(X_B,{\mathbb{Z}})$,
Lemma \ref{lem:Am} yields 
\begin{equation*}
\sum_{i=0}^{m-1} 
\Psi_h(f)(\sigma_A^i(x)) 
=
\sum_{i'=0}^{l_1^{m} (x)-1} f(\sigma_B^{i'}( h(x)))
-
\sum_{j'=0}^{k_1^{m}(x)-1}  
f(\sigma_B^{j'}(h(\sigma_A^{m}(x)))).
\end{equation*}
Hence we have for $m=r,s$
\begin{align*}
  \omega_{\Psi_h(f)}^{r,s}(x) 
=&  \sum_{i=0}^{r-1} \Psi_h(f)(\sigma_A^i(x)) 
- \sum_{j=0}^{s-1} \Psi_h(f)(\sigma_A^j(x)) \\
=
&
\{ 
\sum_{i'=0}^{l_1^{r}(x)-1} f(\sigma_B^{i'}(h(x)))
-
\sum_{j'=0}^{k_1^{r}(x)-1}  
f(\sigma_B^{j'}(h(\sigma_A^{r}(x))))
\} 
\\
-&
\{
\sum_{i'=0}^{l_1^{s}(x)-1} f(\sigma_B^{i'}(h(x)))
-
\sum_{j'=0}^{k_1^{s}(x)-1}  
f(\sigma_B^{j'}(h(\sigma_A^{s}(x)))) \}. 
\end{align*}
The first summand of the first $\{ \cdot \}$ above goes to
\begin{equation*}
 \sum_{i'=0}^{l_1^{r} (x) + k_1^{s}(x)-1} f(\sigma_B^{i'}(h(x))) 
-
 \sum_{i'=l_1^{r}(x)}^{l_1^{r}(x) + k_1^{s}(x)-1} 
f(\sigma_B^{i'}(h(x))). 
\end{equation*}
The first summand of the second $\{ \cdot \}$ above goes to
\begin{equation*}
 \sum_{i'=0}^{l_1^{s} (x) + k_1^{r}(x)-1} 
f(\sigma_B^{i'}(h(x))) 
-
 \sum_{i'=l_1^{s}(x)}^{l_1^{s}(x)) + k_1^{r}(x)-1} 
f(\sigma_B^{i'}(h(x))). 
\end{equation*}
Hence we have
\begin{align*}
  \omega_{\Psi_h(f)}^{r,s}(x) 
= & \{ 
\sum_{i'=0}^{l_1^{r} (x) + k_1^{s}(x)-1} f(\sigma_B^{i'}(h(x)))
 - 
\sum_{i'=0}^{l_1^{s} (x) + k_1^{r}(x)-1} 
f(\sigma_B^{i'}(h(x))) \} \\
- 
&
\{ \sum_{i'=l_1^{r}(x)}^{l_1^{r}(x)) + k_1^{s}(x)-1} 
f(\sigma_B^{i'}(h(x)))
+
\sum_{j'=0}^{k_1^{r}(x)-1}  
f(\sigma_B^{j'}(h(\sigma_A^{r}(x))))
\} \\
 +
&
\{
 \sum_{i'=l_1^{s}(x)}^{l_1^{s}(x) + k_1^{r}(x)-1} 
f(\sigma_B^{i'}(h(x))) 
+
\sum_{j'=0}^{k_1^{s}(x)-1}  
f(\sigma_B^{j'}(h(\sigma_A^{s}(x)))) \}.
\end{align*}
Since
$
l_1^{r}(x) = l_1^{s}(x) + r'
$
and
$
k_1^{r}(x) = k_1^{s}(x) + s',  
$
the first $\{ \cdot \}$ above goes to
\begin{equation}
 \sum_{i=0}^{r'-1}
 f(\sigma_B^{i}(z))
-
\sum_{j=0}^{s'-1}
 f(\sigma_B^{j}(z)). \label{eq:first}
\end{equation}
Since
$
\sigma_B^{l_1^{r}(x)}(h(x)) 
=\sigma_B^{k_1^{r}(x)}(h(\sigma_A^{r}(x))),
$
the second $\{ \cdot \}$ above goes to
\begin{equation}
 \sum_{j'=0}^{k_1^{r}(x) + k_1^{s}(x)-1} 
f(\sigma_B^{j'}(h(\sigma_A^r(x)))).  \label{eq:second}
\end{equation}
Since
$
\sigma_B^{l_1^{s}(x)}(h(x)) 
=\sigma_B^{k_1^{s}(x)}(h(\sigma_A^{s}(x))),
$
the third $\{ \cdot \}$ above goe to
\begin{equation}
 \sum_{j'=0}^{k_1^{s}(x) + k_1^{r}(x)-1} 
f(\sigma_B^{j'}(h(\sigma_A^s(x)))). \label{eq:third} 
\end{equation}
As $\sigma_A^r(x) = \sigma_A^s(x)$,
we have
$
\eqref{eq:second} =\eqref{eq:third}, 
$
so that
$
\omega_{\Psi_h(f)}^{r,s}(x)
= 
\eqref{eq:first}.
$
\end{proof}

We define for $n=1,2,\dots.$
\begin{align*}
c_1(x) & = l_1(x) - k_1(x), \qquad 
c_1^n(x)  = l_1^n(x) - k_1^n(x), \qquad x \in X_A, \\
c_2(y) & = l_2(y) - k_2(y),  \qquad
c_2^n(y)  = l_2^n(y) - k_2^n(y), \qquad y \in X_B.
\end{align*}
The function $c_1$ (resp. $c_2$) is called the cocycle function
for $h$ (resp. $h^{-1}$).
It is clear that the following cocycle conditons hold:
\begin{align*}
c_1^{n+m}(x) & = c_1^n(x) + c_1^m(\sigma_A^n(x)), 
\qquad n, m \in {\mathbb{N}}, \, x \in X_A \\
c_2^{n+m}(y) & = c_2^n(y) + c_2^m(\sigma_B^n(y)), 
\qquad n, m \in {\mathbb{N}}, \, y \in X_B.
\end{align*}

\begin{lemma}\label{lem:sixeq}
The following conditions are equivalent:
\begin{enumerate}
\renewcommand{\labelenumi}{(\roman{enumi})}
\item
$[\Psi_h(f)] \in H^A_+$ for every 
$f \in C(X_B,{\mathbb{Z}})$ with $[f] \in H^B_+$.
\item
$[c_1] \in H^A_+$.
\item
$\omega_{c_1}^{r,s}(x) >0$ 
for $x \in X_A$ such that 
$\sigma_A^r(x) = \sigma_A^s(x)$ and $r-s >0$.
\item
$c_1^{q}(\sigma_A^s(x)) > 0$ 
for $x \in X_A$ such that 
$\sigma_A^r(x) = \sigma_A^s(x)$ and $r-s= q >0$.
\item
$l_1^r(x) + k_1^s(x) > k_1^r(x) + l_1^s(x)$ 
for $x \in X_A$ such that 
$\sigma_A^r(x) = \sigma_A^s(x)$ and $r-s >0$.
\item
$r' > s'$
for $x \in X_A$ such that 
$\sigma_A^r(x) = \sigma_A^s(x)$ and $r-s=q >0$,
where
$r' =l_1^{q}(\sigma_A^s(x)), s' = k_1^{q}(\sigma_A^s(x))$.
\end{enumerate}
\end{lemma}
\begin{proof}
Let 
 $x \in X_A$ satisfy 
$\sigma_A^r(x) = \sigma_A^s(x)$ 
for some 
$r, s \in \Zp$ such that $ r-s= q\in {\mathbb{N}}$.
We then note that 
$r' - s' \ne 0$
by Lemma \ref{lem:6.3}.
The equivalences among (iii), (iv), (v) and (vi) 
come from the following equalities:
\begin{align*}
\omega_{c_1}^{r,s}(x)
= & c_1^r(x) - c_1^s(x) \\
= &  (l_1^r(x) - l_1^s(x)) - ( k_1^r(x) -k_1^s(x)) \\
= & \sum_{i=0}^{q-1} l_1(\sigma_A^{i}(\sigma_A^s(x))) 
-   \sum_{i=0}^{q-1} k_1(\sigma_A^{i}(\sigma_A^s(x))) \\
= & r' - s'  = c_1^q(\sigma_A^s(x)). 
\end{align*}
The equivalence between (ii) and (iii) follows from Lemma \ref{lem:6.2}.
Suppose that the condition (i) holds.
Take the constant function $1_B(y) = 1, y \in X_B$ as a function 
$f \in C(X_B,{\mathbb{Z}})$.
The condition (i) implies that
$[\Psi_h(1_B)] \in H^A_+$.
For $x \in X_A$ we have @
\begin{equation*}
\Psi_h(1_B)(x) 
=
\sum_{i=0}^{l_1(x)-1} 1_B(\sigma_B^{i}(h(x)))
-
\sum_{j=0}^{k_1(x)-1} 1_B(\sigma_B^{i}(h(\sigma_A(x)))) 
= c_1(x)
\end{equation*}
so that 
$[c_1] \in H^A_+$
and the condition (ii) holds.
We finally assume the condition (vi).
For a function $f \in C(X_B,{\mathbb{Z}})$ 
with $[f] \in H^B_+$
and
 $x \in X_A$ with $\sigma_A^r(x) = \sigma_A^s(x)$ and 
$ r-s >0$,
the condition (vi) implies 
$\omega_f^{r',s'}(z) >0$
from $[f] \in H^B_+$
by Lemma \ref{lem:6.2},
where
$z = \sigma_B^{l_1^s(x) + k_1^s(x)}(h(x))$.
Hence the equality
\eqref{eq:oPhfx}
implies
$\omega_{\Psi_h(f)}^{r,s}(x) > 0$ 
so that
$[\Psi_h(f)] \in H^A_+$
by Lemma \ref{lem:6.2} again.
This implies the condition (i).
\end{proof}


In the rest of the section,
we will show that $[c_1] \in H_+^A$ always holds. 
\begin{definition} \label{defn:attracting}
For $r,s \in \Zp$, an eventually periodic point $x \in X_A$
is said to be $(r,s)$-{\it attracting} 
if it satisfies the following two conditions:
\begin{enumerate}
\renewcommand{\labelenumi}{(\roman{enumi})}
\item
$\sigma_A^r(x) = \sigma_A^s(x).$ 
\item
For any clopen neighborhood $W \subset X_A$ of $x$,
there exist clopen sets $U, V \subset X_A$ and a homeomorphism
$\varphi: V \longrightarrow U$
 such that
{\begin{enumerate}
\item $x \in U \subset V \subset W$.
\item $\varphi(x) = x$.
\item $\sigma_A^r(\varphi(w)) = \sigma_A^s(w)$ for all $w \in V$. 
\item $\lim_{n \to\infty} \varphi^n(w) = x$  for all $w \in V$. 
\end{enumerate}}
\end{enumerate}
\end{definition}
For a word $\mu = \mu_1\cdots\mu_k \in B_k(X_A)$,
denote by $U_\mu \subset X_A$
the cylinder set
\begin{equation*}
U_\mu = 
\{ (x_n)_{n \in {\mathbb{N}}} \in X_A 
\mid x_1=\mu_1,\dots,x_k = \mu_k \}.
\end{equation*}
\begin{lemma}
An eventually periodic point is $(r,s)$-attracting
for some $r,s \in \Zp$.
\end{lemma}
\begin{proof}
Let $x \in X_A$ be  an eventually periodic point
such that $\sigma_A^r(x) = \sigma_A^s(x)$
with $r \ne s$.
We may assume that 
$r >s$.
Put
the words
$\nu = x_{[1,s]}, \xi = x_{[s+1,r]}, \mu = x_{[1,r]}$.
One has
$\mu = \nu \xi$ and
$x = \nu \xi\xi\xi\cdots$.
For a clopen neighborhood $W \subset X_A$ of $x$,
there exist
$L \in {\mathbb{N}}$ such that
 by putting
 $\bar{\nu} =\nu \overbrace{\xi\cdots\xi}^L$
 and
 $\bar{\mu} =\mu \overbrace{\xi\cdots\xi}^{L+1}$,
one has 
 $ x \in U_{\bar{\mu}} \subset U_{\bar{\nu}} \subset W$.
 We set 
 $V = U_{\bar{\nu}}$ and $U = U_{\bar{\mu}}.$
Define 
$\varphi:V\longrightarrow U$
by substituting $\bar{\mu}$
for the left most word $\bar{\nu}$ of elements of $V$.
It is a homeomorphism
from $V$ to $U$ such that
$\varphi(x) = x$.
Since
$|\nu| = s, |\mu| =r$,
the equalities
$\sigma_A^r(\varphi(w)) = \sigma_A^s(w)$ for all $w \in V$
hold. 
As 
$
\varphi^n(w) 
$
begins with
$\nu \overbrace{\xi\cdots\xi}^{L+n}$,
we have
 $\lim_{n \to\infty} \varphi^n(w) = x$  for all $w \in V$. 
 \end{proof}
\begin{lemma}
If an eventually periodic point is $(r,s)$-attracting,
then $r >s$.
\end{lemma}
\begin{proof}
Let $x \in X_A$ be $(r,s)$-attracting.
For $W =X_A$, 
take clopen sets $U, V \subset X_A$ and a homeomorphism
$\varphi: V \longrightarrow U$
satisfying the conditions (ii) of Definition \ref{defn:attracting}.
We note that the matrix $A$ satisfies condition (I) 
in the sense of \cite{CK}
so that $X_A$ is homeomorphic to a Cantor set.
Assume that $r \le s$.
We have two cases.

Case 1 : $r=s$.

Take $w \in V$
such that
$w_{[r+1,\infty)} \ne x_{[r+1,\infty)}$.
By the condition (c) of (ii) in Definition \ref{defn:attracting},
 one sees 
$\sigma_A^r(\varphi^n(w)) = \sigma_A^r(w), n \in {\mathbb{N}}$
so that
$\lim_{n\to\infty}\sigma_A^r(\varphi^n(w)) = \sigma_A^r(w)$,
which contradicts to the condition
$\lim_{n\to\infty}\varphi^n(w) = x$
with
$w_{[r+1,\infty)} \ne x_{[r+1,\infty)}$.

Case 2 : $r <s$.

Put
$q = s - r \in {\mathbb{N}}.$
For all $w \in V$, we have
$\varphi(w)_{[r+1,\infty)} =w_{[s+1,\infty)}$.
As 
$\varphi^n(w) \in V$ for $n \in {\mathbb{N}}$,
 we have
$\varphi^n(w)_{[r+1,\infty)} =w_{[s+(n-1)q+1,\infty)}$
so that 
$\varphi^n(w)_{[r+1,r+q]} =w_{[s+(n-1)q+1,s+nq]}$ for all 
$n \in {\mathbb{N}}$.
Hence
$\lim_{n\to\infty} \varphi^n(w)$ does not exist 
unless $\sigma_A^r(w)$ is $q$-periodic.
There exists a point $w \in V$ 
which is not an eventually periodic point,
a contradiction.

Therefore the above two cases do not occur. 
\end{proof}

\begin{lemma}
If $x$ is $(r,s)$-attracting,
then
$h(x)$ is 
$(l_1^r(x) +k_1^s(x),   k_1^r(x) + l_1^s(x) )$-attracting.
\end{lemma}
\begin{proof}
For a clopen neighborhood $W' \subset X_B$ of $h(x)$, 
put 
a clopen neighborhood 
$W = h^{-1}(W') \subset X_A$ of $x$.
Since the functions 
$l_1^r, \,
k_1^s, \,
k_1^r, \,
l_1^s
$
are all continuous,
one may take 
$W$ small enough such that 
\begin{equation*}
l_1^r(w) = l_1^r(x), \quad
k_1^s(w) = k_1^s(x), \quad
k_1^r(w) = k_1^r(x), \quad
l_1^s(w) = l_1^s(x) \quad
\end{equation*}
for all $w \in W$.
Put
$r' =l_1^r(x) + k_1^s(x),  
s'  =k_1^r(x) + l_1^s(x). 
$
By Lemma \ref{lem:taileq},
one has
$\sigma_B^{r'}(h(x)) = \sigma_B^{s'}(h(x)).
$
Take 
clopen sets $U, V \subset X_A$ and a homeomorphism
$\varphi: V \longrightarrow U$
satisfying the condition (ii) of Definition \ref{defn:attracting}.
We set 
$U' = h(U), V' = h(V)$ of $X_B$
and a homeomorphism
$\varphi' = h \circ \varphi \circ h^{-1} |_{V'}:V' \longrightarrow U'$.
They satisfy
$h(x) \in U' \subset V' \subset W'.$
As
$\sigma_A^r (\varphi(w)) = \sigma_A^s(w)$
for $w \in V$,
Lemma \ref{lem:taileq}
ensures us
\begin{equation}
\sigma_B^{l_1^r(\varphi(w)) + k_1^s(w)}(h(\varphi(w))) 
= \sigma_B^{k_1^r(\varphi(w)) + l_1^s(w)}(h(w)) \label{eq:lrks}
\end{equation}
for $w \in V$.
Since
$\varphi(w) \in V$
for $w \in V$,
one sees that
$
l_1^r(\varphi(w)) = l_1^r(x),  
k_1^s(w)  = k_1^s(x), 
k_1^r(\varphi(w)) = k_1^r(x),  
l_1^s(w) = l_1^s(x).
$
Hence the equality \eqref{eq:lrks}
goes to
\begin{equation*}
\sigma_B^{r'}(h(\varphi(w))) 
= \sigma_B^{s'}(h(w)) 
\quad
\text{ and hence }
\quad
\sigma_B^{r'}(\varphi'(h(w))) 
= \sigma_B^{s'}(h(w)) 
\end{equation*}
for all $w \in V$.
The equality
$\lim_{n\to\infty}{\varphi'}^n(w) = h(x)$
for $w \in V$ is easily verified,
so that 
$h(x)$ is $(r',s')$-attracting.
\end{proof}
\begin{corollary} \label{cor:7.8}
Keep the above notations.
If $x \in X_A$ satisfies
$\sigma_A^r(x) = \sigma_A^s(x)$
for some $r > s$,
then  
$l_1^r(x) +k_1^s(x) > k_1^r(x) + l_1^s(x) $
and hence
$c_1^q(\sigma_A^s(x)) >0$ where $q = r-s$.
\end{corollary}
By using the above corollary with Lemma \ref{lem:sixeq},
we reach the following proposition and theorem.
\begin{proposition}
The class $[c_1]$
of the cocycle function
$c_1(x) = l_1(x) - k_1(x)$ for $x \in X_A$
in $H^A$
gives rise to a positive element
in the ordered cohomology group
$(H^A, H^A_+)$ which is an order unit in $(H^A, H^A_+)$.
\end{proposition}
Therefore we have
\begin{theorem}[cf. {\cite[Theorem 3.5]{MM}}] \label{thm:ordercoho}
Let $h$ 
be a homeomorphism from $X_A$ to $X_B$
which gives rise to a continuous orbit equivalence 
between the one-sided topological Markov shifts
$(X_A,\sigma_A)$ and $(X_B,\sigma_B)$.
Then
there exist isomorphisms
$\Psi_h:C(X_B,\Z) \longrightarrow C(X_A,\Z)$
and
$\Psi_{h^{-1}}:C(X_A,\Z) \longrightarrow C(X_B,\Z)$
which are inverses to each other such that 
\begin{enumerate}
\renewcommand{\labelenumi}{(\roman{enumi})}
\item 
$\Psi_h(1_B) (x) = l_1(x) - k_1(x)$ for $x \in X_A$,
\item 
$\Psi_{h^{-1}}(1_A) (y) = l_2(y) - k_2(y)$ for $y \in X_B$,
\item
$[\Psi_h(f)] \in H_+^A$ for $[f] \in H_+^B$, 
\item
$[\Psi_{h^{-1}}(g)] \in H_+^B$ for $[g] \in H_+^A$ 
\end{enumerate}
so that
$\Psi_h$ induces an isomorphism
$\bar{\Psi}_h: 
(H^B, H_+^B)\longrightarrow (H^A, H_+^A)$
of the ordered cohomology groups
$(H^A, H_+^A)$ and 
$(H^B, H_+^B)$
as ordered groups.
\end{theorem}

\section{Periodic points and zeta functions} 
Continuous orbit equivalence between one-sided topological Markov shifts 
preserves their eventually periodic points.
Eventually periodic points of a one-sided topological Markov shift
naturally yield periodic points of the two-sided topological Markov shift. 
In this section,
we will study periodic points of two-sided topological Markov shifts
whose one-sided topological Markov shifts
are
continuously orbit equivalent.
Recall that 
$(\bar{X}_A,\bar{\sigma}_A)$
stands for the two-sided topological Markov shift
for matrix $A$.
For $p \in \N$,
put the set of periodic points
\begin{equation*}
\Per_p(\bar{X}_A) = \{\bar{x}= (x_n )_{n \in \Z} 
\in \bar{X}_A
\mid
\bar{\sigma}_A^p( \bar{x}) = \bar{x} \}
\end{equation*}
and
$\Per_*(\bar{X}_A) 
=\cup_{p=1}^\infty \Per_p(\bar{X}_A).
$
For $\bar{x} \in \Per_*(\bar{X}_A)$,
the subset
$
\gamma = \{ \bar{\sigma}_A^n(\bar{x}) \in \bar{X}_A \mid n \in \Z\}
$
of
$\bar{X}_A$
is called 
the periodic orbit of $\bar{x}$ under $\bar{\sigma}_A$. 
We call the cardinality $|\gamma|$ of $\gamma$
the period of $\gamma$ which is the least period of $\bar{x}$
under $\bar{\sigma}_A$.  
If $\bar{x} \in \Per_p(\bar{X}_A)$,
then $p = k |\gamma|$ for some $k \in \N$.
Let $P_{orb}(\bar{X}_A)$
be the set of periodic orbits of
$(\bar{X}_A,\bar{\sigma}_A)$.
Denote by
$\pi_A: \bar{X}_A \longrightarrow X_A$
the restriction of $\bar{X}_A$ to $X_A$ defined by
$\pi_A( (x_n)_{n \in \Z}) 
=(x_n)_{n \in \N}.
$
We are assuming that $h$ 
is a homeomorphism from $X_A$ to $X_B$
which gives rise to
a continuous orbit equivalence between
$(X_A,\sigma_A)$ and $(X_B,\sigma_B)$.
Recall that $c_1^p \in C(X_A,{\mathbb{Z}})$
for $p\in {\mathbb{N}}$ is the cocycle function
defined by
$c_1^p(x) = l_1^p(x) - k_1^p(x), x \in X_A$.
\begin{lemma}
There exists a map 
$
\psi_{h}:\Per_*(\bar{X}_A)\longrightarrow \Per_*(\bar{X}_B) 
$
such that for $\bar{x} \in \Per_p(\bar{X}_A)$,
\begin{align}
\sigma_B^{k_1^p(x)}(\pi_B(\psi_h(\bar{x}))) 
& = \sigma_B^{k_1^p(x)}(h(\pi_A(\bar{x}))), \label{eq:8.1}\\ 
\bar{\sigma}_B^{k_1(x)}(\psi_h(\bar{\sigma}_A(\bar{x})))
& = \bar{\sigma}_B^{l_1(x)}(\psi_h(\bar{x})), \label{eq:8.2} \\
\bar{\sigma}_B^{c_1^p(x)} (\psi_h (\bar{x})) 
& = \psi_h (\bar{x}), \label{eq:8.3}
\end{align}
where
$x = \pi_A(\bar{x})$.
\end{lemma}
\begin{proof}
For
$\bar{x}
\in \Per_p(\bar{X}_A),
$ 
put
$x =\pi_A(\bar{x})$
which 
is a $p$-periodic point of
 $X_A$. 
By \eqref{eq:3.10} and Corollary \ref{cor:7.8},
$h(x)$ is an eventually periodic point of $X_B$
such that 
$\sigma_B^{k_1^p(x)}(h(x))$ is a 
$c_1^p(x)$-periodic point.
There exists a unique element
$\bar{y}$
in $\bar{X}_B$
satisfying 
$\bar{\sigma}_B^{c_1^p(x)}(\bar{y}) = \bar{y}$
and
$
\pi_B(\bar{\sigma}_B^{k_1^p(x)}(\bar{y}))
=\sigma_B^{k_1^p(x)}(h(x))$.
That is a unique extension of the $c_1^p(x)$-periodic point
$\sigma_B^{k_1^p(x)}(h(x))$ 
to a two-sided sequence in $\bar{X}_B$.
Define
$\psi_{h}(\bar{x}) = \bar{y} \in \bar{X}_B$ 
so that 
\begin{equation}
\psi_h(\bar{x})_{[k_1^p(\sigma_A(x)),\infty)}
 = h(\pi_A(\bar{x}))_{[k_1^p(\sigma_A(x)),\infty)},
 \label{eq:psik1}
\end{equation}
and hence the equalities
\eqref{eq:8.1} and 
\eqref{eq:8.3} are obvious.
By 
$k_1^p(\sigma_A(x)) = k_1^p(x)$
and
\eqref{eq:psik1},
we have
\begin{align*}
\sigma_B^{k_1(x)} (h(\sigma_A(x)))_{[k_1^p(\sigma_A(x)),\infty)}
& = h(\pi_A(\bar{\sigma}_A(\bar{x})))_{[k_1^p(\sigma_A(x))+k_1(x),\infty)} \\
& = \psi_h(\bar{\sigma}_A(\bar{x}))_{[k_1^p(\sigma_A(x))+k_1(x),\infty)} \\
& = \bar{\sigma}_B^{k_1(x)}(\psi_h(\bar{\sigma}_A(\bar{x})))_{[k_1^p(x),\infty)} \\
\intertext{and}
\sigma_B^{l_1(x)}(h(x))_{[k_1^p(\sigma_A(x)),\infty)}
& =h(\pi_A(\bar{x}))_{[k_1^p(\sigma_A(x)) +l_1(x),\infty)}\\
& =\psi_h(\bar{x})_{[k_1^p(x) +l_1(x),\infty)}\\
& =\bar{\sigma}_B^{l_1(x)}(\psi_h(\bar{x}))_{[k_1^p(x),\infty)}
\end{align*}
so that the identity \eqref{eq:orbiteqx} implies
\begin{equation*}
\bar{\sigma}_B^{k_1(x)}(\psi_h(\bar{\sigma}_A(\bar{x})))_{[k_1^p(x),\infty)}
=\bar{\sigma}_B^{l_1(x)}(\psi_h(\bar{x}))_{[k_1^p(x),\infty)}
\end{equation*}
As both 
$\bar{\sigma}_B^{k_1(x)}(\psi_h(\bar{\sigma}_A(\bar{x})))$
and
$\bar{\sigma}_B^{l_1(x)}(\psi_h(\bar{x}))$
are periodic, we obtain
\eqref{eq:8.2}.
Thus 
$
\psi_{h}:\Per_*(\bar{X}_A)\longrightarrow \Per_*(\bar{X}_B)
$
satisfies the desired properties.
\end{proof}
By \eqref{eq:8.2}, the above map
$
\psi_{h}:\Per_*(\bar{X}_A)\longrightarrow \Per_*(\bar{X}_B)
$
preserves each orbit of periodic points so that it
induces a map
\begin{equation*}
\xi_{h}:P_{orb}(\bar{X}_A)\longrightarrow P_{orb}(\bar{X}_B)
\end{equation*}
such that
$\xi_h(\gamma) =
\{ \bar{\sigma}_B^m(\psi_h(\bar{\sigma}_A^n(\bar{x}))) \mid n, m \in \Z\} \subset \bar{X}_B$
for
$\gamma =\{ \bar{\sigma}_A^n(\bar{x}) \mid n \in \Z\}\subset \bar{X}_A$.
We also have a map
$\psi_{h^{-1}}:\Per_*(\bar{X}_B)\longrightarrow \Per_*(\bar{X}_A)
$
and the induced map
$
\xi_{h^{-1}}:P_{orb}(\bar{X}_B)\longrightarrow P_{orb}(\bar{X}_A)
$
for the inverse $h^{-1}: X_B \longrightarrow X_A$
of $h$.
\begin{lemma}
For $\bar{x} \in \Per_p(\bar{X}_A)$, 
put
$x = \pi_A(\bar{x}), q= c_1^p(x), n=k_1^p(x),
\bar{y}=\psi_h(\bar{x}), y = \pi_B(\bar{y})$.
Then we have
\begin{equation*}
\sigma_A^{l_2^n(y)+k_2^q(y)}(\pi_A(\psi_{h^{-1}}(\psi_h(\bar{x}))))
=\sigma_A^{l_2^n(y)+k_2^q(y)}(\pi_A(\bar{x}))
\end{equation*}
so that
$\xi_{h^{-1}}\circ \xi_h = \id$ on $P_{orb}(\bar{X}_A)$,
and similarly
$\xi_h \circ \xi_{h^{-1}}= \id$ on $P_{orb}(\bar{X}_B)$.
\end{lemma}
\begin{proof}
By \eqref{eq:8.1} for
$h^{-1}$ and $\bar{y}=\psi_h(\bar{x})$,
we have
\begin{equation*}
\sigma_A^{k_2^q(y)}(\pi_A(\psi_{h^{-1}}(\psi_h(\bar{x})))) 
 = \sigma_A^{k_2^q(y)}(h^{-1}(\pi_B(\psi_h(\bar{x})))).
\end{equation*}
As in \eqref{eq:8.1}, we have
$
\sigma_B^{n}(\pi_B(\psi_h(\bar{x}))) 
 = \sigma_B^{n}(h(x)).
$
As
\begin{align*}
\sigma_A^{l_2^n(y)}(h^{-1}(\pi_B(\psi_h(\bar{x})))) 
& =
\sigma_A^{k_2^n(y)}(h^{-1}(\sigma_B^n(\pi_B(\psi_h(\bar{x}))))) \\
& =
\sigma_A^{k_2^n(y)}
(h^{-1}(\sigma_B^n(h(x)))) 
=
\sigma_A^{l_2^n(y)}(\pi_A(\bar{x})),
\end{align*}
we see that
\begin{align*}
 & \sigma_A^{l_2^n(y) + k_2^q(y)}(\pi_A(\psi_{h^{-1}}(\psi_h(\bar{x})))) \\
=& \sigma_A^{k_2^q(y)}(\sigma_A^{l_2^n(y)}(h^{-1}(\pi_B(\psi_h(\bar{x})))))
= \sigma_A^{l_2^n(y) + k_2^q(y)}(\pi_A(\bar{x})).
\end{align*}
Hence
$\pi_A(\psi_{h^{-1}}(\psi_h(\bar{x})))$
and
$\pi_A(\bar{x})$
are in the same orbit in $X_A$
so that
$\psi_{h^{-1}}(\psi_h(\bar{x}))$
and
$\bar{x}$
are in the same orbit in $\bar{X}_A$.
Therefore we see that
$\xi_{h^{-1}}\circ \xi_h = \id$ on $P_{orb}(\bar{X}_A)$
and similarly
$\xi_h \circ \xi_{h^{-1}}= \id$ on $P_{orb}(\bar{X}_B)$.
\end{proof}
The above lemma says that continuous orbit equivalence between one-sided topological Markov shifts
yields a bijective correspondence
between the sets of periodic orbits of 
their two-sided topological Markov shifts.
\begin{lemma}\label{lem:factor}
Let $x \in X_A$
satisfy $\sigma_A^r(x) = \sigma_A^s(x)$
such that
$r -s = q = n p \in \N$ for some $n \in \N$,
where
$p$ is the least period of $\sigma_A^s(x)$.
Then 
\begin{equation*}
c_1^q(\sigma_A^s(x)) = n \cdot c_1^p(\sigma_A^s(x)).
\end{equation*}
\end{lemma}
\begin{proof}
As $
\sigma_A^{s+j}(x) = \sigma_A^{s+ip+j}(x)$
for $j=0,1,\dots,p-1$ and $i= 0,1,\dots,n-1$,
we have
$$
c_1^q(\sigma_A^s(x)) 
=\sum_{m=0}^{q-1} c_1(\sigma_A^{s+m}(x))
=n \sum_{j=0}^{p-1} c_1(\sigma_A^{s+j}(x))
= n\cdot c_1^p(\sigma_A^s(x)).
$$
\end{proof}
\begin{lemma} \label{lem:rminuss}
Let $x \in X_A$
satisfy $\sigma_A^r(x) = \sigma_A^s(x)$
such that
$r -s = q \in \N$.
Put
$
z = \sigma_B^{l_1^s(x) +k_1^s(x)}(h(x) )\in X_B, \,
r' = l_1^q(\sigma_A^s(x)), \,  
s' = k_1^q(\sigma_A^s(x)), \,
q' = r' -s'.
$
Then we have
\begin{equation*}
c_2^{q'}(\sigma_B^{s'}(z)) = r-s.
\end{equation*}
\end{lemma}
\begin{proof}
We note that by Corollary \ref{cor:7.8},
$q' = l_1^r(x) + k_1^s(x) - (k_1^r(x) + l_1^s(x) ) >0.$
It then follows that by Lemma \ref{lem:6.3} and \eqref{eq:oPhfx}, \eqref{eq:first},
\begin{align*}
c_2^{q'}(\sigma_B^{s'}(z)) 
& = l_2^{r'}(z) - l_2^{s'}(z) - (k_2^{r'}(z) - k_2^{s'}(z)) \\ 
& = \omega_{\Psi_h(l_2)}^{r,s}(x) -  \omega_{\Psi_h(k_2)}^{r,s}(x) 
  =  \omega_{\Psi_h(c_2)}^{r,s}(x).
\end{align*}
As
$c_2 =\Psi_{h^{-1}}(1_A)$
and
$\Psi_h(\Psi_{h^{-1}}(1_A)) =1_A$,
we have
$$
c_2^{q'}(\sigma_B^{s'}(z)) 
= \omega_{1_A}^{r,s}(x) = r-s.
$$
\end{proof}
\begin{lemma}
For $\gamma \in P_{orb}(\bar{X}_A)$
with
$|\gamma|=p$,
take $\bar{x} \in \gamma$ and put
$x = \pi_A(\bar{x})$.
Then we have
\begin{equation*}
| \xi_h(\gamma)| = c_1^p(x) (= \omega_{c_1}^{r,s}(x)).
\end{equation*}
\end{lemma}
\begin{proof}
The least period of $x$ is $p$.
Put
$z = h(x)$
and
$
r' = l_1^p(x), 
s' = k_1^p(x), 
q' = r' - s'.
$
By Corollary \ref{cor:7.8} for $r=p$ and $s=0$, 
we know that $r' - s' = q'= c_1^p(x) >0$. 
Denote by $p'$ the least eventual period of $z$.
By Lemma \ref{lem:6.3},
$z$ has an eventual period 
$r' - s' = q'$.
Hence we have
\begin{equation}
c_1^p(x) =q' = n' \cdot p' \qquad \text{ for some } n' \in \N.
\label{eq:maru1}
\end{equation}
Since $|\xi_h(\gamma)|$ coincides with $p'$,
it suffices to prove that
$n' =1$.
As $p'$ is the least period of 
$\sigma_B^{s'}(z)$,
by applying Lemma \ref{lem:factor}
for $\sigma_B^{r'}(z) = \sigma_B^{s'}(z), r'- s'=q'=n' p'$ 
we have
\begin{equation}
c_2^{q'}(\sigma_B^{s'}(z)) 
= n' \cdot 
c_2^{p'}(\sigma_B^{s'}(z)). 
\label{eq:maru2}
\end{equation}
By applying Lemma \ref{lem:rminuss}
for $r=p, s=0, z=h(x), r' = l_1^p(x), s' = k_1^p(x), r'-s' = q'$,
we have
\begin{equation}
c_2^{q'}(\sigma_B^{s'}(z)) 
= p. 
\label{eq:maru2prime}
\end{equation}
Since 
$\sigma_B^{s'+p'}(z)= \sigma_B^{s'}(z)$,
by puting
$
x' = \sigma_A^{l_2^{s'}(z)+k_2^{s'}(z)}(x), 
$
we have
\begin{align*}
\sigma_A^{l_2^{p'}(\sigma_B^{s'}(z))}(x')
&= \sigma_A^{l_2^{p'}(\sigma_B^{s'}(z)) + l_2^{s'}(z)+k_2^{s'}(z)}(h^{-1}(z)) \\
&= \sigma_A^{k_2^{s'}(z)}(  \sigma_A^{l_2^{s'+p'}(z)}(h^{-1}(z)) \\
&= \sigma_A^{k_2^{s'}(z)}(  \sigma_A^{k_2^{s'+p'}(z)}(h^{-1}(\sigma_B^{s'+p'}(z))) \\
&= \sigma_A^{k_2^{s'+p'}(z)}(  \sigma_A^{k_2^{s'}(z)}(h^{-1}(\sigma_B^{s'}(z))) \\
&= \sigma_A^{k_2^{p'}(\sigma_B^{s'}(z)) +k_2^{s'}(z) +l_2^{s'}(z)}(h^{-1}(z))) \\
&= \sigma_A^{k_2^{p'}(\sigma_B^{s'}(z))}(x')
\end{align*}
so that
\begin{equation}
\sigma_A^{l_2^{p'}(\sigma_B^{s'}(z))}(x')
=\sigma_A^{k_2^{p'}(\sigma_B^{s'}(z))}(x'). \label{eq:maru5}
\end{equation}
The least period of $x'$ is equal to 
that of $x $
which is $p$.
By \eqref{eq:maru5},
we have 
\begin{equation*}
c_2^{p'}(\sigma_B^{s'}(z)) 
=
l_2^{p'}(\sigma_B^{s'}(z)) - k_2^{p'}(\sigma_B^{s'}(z)) 
= m' \cdot p
\qquad \text{ for some } m' \in \Z.
\end{equation*}
By \eqref{eq:maru2} and \eqref{eq:maru2prime},
we see
$$
p 
= c_2^{q'}(\sigma_B^{s'}(z)) 
=n' \cdot c_2^{p'}(\sigma_B^{s'}(z)) 
=n' \cdot m' \cdot p. 
$$ 
We thus conclude that $n' = m' = 1$
so that
\eqref{eq:maru1} implies $c_1^p(x) = p' = |\xi_h(\gamma)|$. 
\end{proof}

For $f \in C(X_A,\Z)$ with $[f] \in H^A_+$
and
a finite periodic orbit $O \subset X_A$,
which has a point $x \in X_A$ and $p \in \N$
such that 
$x = \sigma_A^p(x)$ and 
$O = \{ x,\sigma_A(x),\dots, \sigma_A^{p-1}(x)\}$,
we set
\begin{equation*}
\beta_O([f]) = \sum_{i=0}^{p-1}f(\sigma_A^i(x)).
\end{equation*} 
We may naturally identify 
finite periodic orbits of
$(X_A,\sigma_A)$ with 
finite periodic orbits of
$(\bar{X}_A,\bar{\sigma}_A)$. 
If for $\gamma \in P_{orb}(\bar{X}_A)$ 
and $f=1_A$, one sees that
$$
\beta_\gamma([1_A]) = \sum_{x \in \gamma}1_A(x) = |\gamma|
:\quad \text{the length of periodic orbit of } \gamma.
$$ 
\begin{lemma}
For a periodic point $x$ in $X_A$ 
with least period $p \in \N$,
put
$\gamma= \{ x,\sigma_A(x),\dots, \sigma_A^{p-1}(x)\}$
the orbit of $x$.
Then we have
\begin{equation*}
\beta_\gamma([c_1]) = |\xi_h(\gamma)|.
\end{equation*}
\end{lemma}
\begin{proof}
In the preceding lemma,
one knows that
$
|\xi_h(\gamma)| = 
c_1^p(x).
$
Since
$
\beta_\gamma([c_1]) 
 =\sum_{i=0}^{p-1}c_1(\sigma_A^i(x)) =c_1^p(x),
$ 
one has
$
\beta_\gamma([c_1]) = |\xi_h(\gamma)|.
$
\end{proof}
Denote by 
$|\Per_n(\bar{X}_A)|$
the cardinality of the set
$\Per_n(\bar{X}_A)$ of $n$-periodic points of
$(\bar{X},\bar{\sigma}_A)$.
The zeta function $\zeta_A(t)$ 
for 
$(\bar{X},\bar{\sigma}_A)$
is defined by
\begin{equation*}
\zeta_A (t)
= \exp\left(
\sum_{n=1}^\infty\frac{t^n}{n}|\Per_n(\bar{X}_A)|
\right).
\end{equation*}
It has the following Euler product formula (see \cite[Section 6.4]{LM})
\begin{equation*}
\zeta_A(t) = \prod_{\gamma \in P_{orb}(\bar{X}_A)}
(1 - t^{|\gamma|})^{-1}.
\end{equation*}
In \cite[p. 176]{BH},
the zeta function of an order unit in  ordered cohomology group
has been studied related to flow equivalence of topological Markov shifts.
In our situation, 
the class $[c_1]$ of the cocycle function
$c_1(x) =  l_1(x) - k_1(x),  x \in X_A
$
gives rise to 
an order unit in the ordered group
$(H^A,H^A_+)$.
Hence the zeta function 
$\zeta_{[c_1]}(t) $ 
for the order unit $[c_1]$
may be defined in the sense of \cite{BH}, 
which goes to 
\begin{equation*}
\zeta_{[c_1]}(t) = \prod_{\gamma \in P_{orb}(\bar{X}_A)}
(1 - t^{\beta_\gamma([c_1])})^{-1}.
\end{equation*}
We note that by putting $t =e^{-s}$,
the zeta function  
$\zeta_{[c_1]}(t) $ coincides with the following zeta function
$\zeta_{\bar{\sigma}_A, \bar{c}_1}(s)$
so called the dynamical zeta function with potential
$\bar{c}_1(\bar{x}) = c_1 \circ \pi_A(\bar{x}), 
\bar{x} \in \bar{X}_A$
\begin{equation}
\zeta_{\bar{\sigma}_A,\bar{c}_1}(s)
= \exp\{
\sum_{n=1}^\infty\frac{1}{n}\sum_{\bar{x}\in \Per_n(\bar{X}_A)}
\exp( -s \sum_{k=0}^{n-1}\bar{c}_1(\bar{\sigma}_A^k(\bar{x}))) \} \label{eq:dynamiczeta}
\end{equation}
(see \cite{PP}, \cite{Ruelle1978}, \cite{Ruelle2002}).
If $[c_1] = [1_A]$ in $H^A_+$
so that the unital ordered groups
$(H^A,H^A_+,[1_A])$ and
$(H^B,H^B_+,[1_B])$
are isomorphic,
then
$\beta_\gamma([c_1]) = |\gamma|$ so that 
\begin{equation*}
\zeta_{[c_1]}(t) = \zeta_A(t).
\end{equation*}
As 
$\xi_h:P_{orb}(\bar{X}_A) \longrightarrow
P_{orb}(\bar{X}_B) 
$
is bijective,
the preceding lemma implies that the equality 
\begin{equation*}
\zeta_{[c_1]}(t) = \prod_{\gamma \in P_{orb}(\bar{X}_A)}
(1 - t^{|\xi_h(\gamma)|})^{-1}
=\zeta_B(t)
\end{equation*}
holds.
We thus have the following theorem which describes structure of
periodic points of the two-sided topological Markov shifts
in continuously orbit equivalent one-sided topological Markov shifts.   
\begin{theorem}\label{thm:zeta}
Suppose that
one-sided topological Markov shifts
$(X_A,\sigma_A)$ and $(X_B,\sigma_B)$
are continuously orbit equivalent 
through  a homeomorphism
 $h$ from $X_A$ to $X_B$ 
satisfying \eqref{eq:orbiteqx} and \eqref{eq:orbiteqy}.
Let
$\zeta_A(t)$
and
$\zeta_B(t)$
be the zeta functions for 
their two-sided topological Markov shifts
$(\bar{X}_A,\bar{\sigma}_A)$ 
and
$(\bar{X}_B,\bar{\sigma}_B)$ 
respectively.
Then the zeta functions
$\zeta_A(t)$
and
$\zeta_B(t)$
coincide with
 the zeta functions
$\zeta_{[c_2]}(t)$ and
$\zeta_{[c_1]}(t)$
for
the cocycle functions
$
c_2(y) =  l_2(y) - k_2(y),  y \in X_B
$
and
$
c_1(x) = l_1(x) - k_1(x), x \in X_A
$
respectively.
That is
\begin{equation*}
\zeta_A(t) = \zeta_{[c_2]}(t), \qquad
\zeta_B(t) = \zeta_{[c_1]}(t).
\end{equation*}
\end{theorem}
One may easily see that the condition
 $[c_1] =[1_A]$ in $H^A$
 implies
$[c_2] =[1_B]$ in $H^B$.
Hence if 
$[c_1] =[1_A]$
or
$[c_2] =[1_B]$,
then $\zeta_A(t) = \zeta_B(t)$.
This shows that the two-sided Markov shifts
$(\bar{X}_A,\bar{\sigma}_A)$ and $(\bar{X}_B,\bar{\sigma}_B)$ 
are almost conjugate 
(cf. \cite[Theorem 9.3.2]{LM}).

\begin{corollary}[cf. {\cite[Theorem 3.5]{MM}}]
Suppose that
one-sided topological Markov shifts
$(X_A,\sigma_A)$ and $(X_B,\sigma_B)$
are continuously orbit equivalent.
Then we have
$\det(\id - A) = \det(\id -B)$. 
\end{corollary}
\begin{proof}
We know $\zeta_{[c_1]}(t) = 
\zeta_{\bar{\sigma}_A,\bar{c}_1}(s)$
for $t = e^{-s}$.
As $\zeta_B(t) = \zeta_{[c_1]}(t)$,
by putting $s=0$ in \eqref{eq:dynamiczeta}, 
we get
$\det(\id - B) = \det(\id -A)$. 
\end{proof}
The above corollary has been already shown in \cite[Theorem 3.5]{MM}
by using \cite{BH} and \cite{PS}.
Our proof in this paper is direct without using their results.

\section{Invariant measures}
In this section, we will show that the set of
$\sigma_A$-invariant measures on $X_A$ 
is invariant under the continuous orbit equivalence class of one-sided topological Markov shift
$(X_A,\sigma_A)$.
Throughout the section,
a measure means  a regular Borel measure.
We denote by
$M(X_A,\sigma_A),M(X_A,\sigma_A)_+$
and
$P(X_A,\sigma_A)$
the set of $\sigma_A$-invariant measures on $X_A$,
the set of $\sigma_A$-invariant positive measures on $X_A$ and
the set of $\sigma_A$-invariant probability measures on $X_A$
respectively.
We identify a regular Borel measure on $X_A$ 
with a continuous linear functional on the commutative $C^*$-algebra
$C(X_A,{\mathbb{C}})$ of ${\mathbb{C}}$-valued continuous functions on $X_A$.
By the identification, one may write
\begin{align*}
M(X_A,\sigma_A)
& = \{ \varphi \in C(X_A,{\mathbb{C}})^* \mid
\varphi(f \circ \sigma_A) = \varphi(f) \text{ for all } f \in C(X_A,{\mathbb{C}}) \}, \\ 
M(X_A,\sigma_A)_+
& = \{ \varphi \in M(X_A,\sigma_A) \mid
\varphi(f) \ge 0 \text{ for all } f \in C(X_A,{\mathbb{C}}) \text{ with }
f \ge 0 \}, \\ 
P(X_A,\sigma_A)
& = \{ \varphi \in M(X_A,\sigma_A)_+ \mid
\varphi(1) = 1 \}. 
\end{align*}
Let
$h$ be a homemorphism
which gives rise to a continuous orbit equivalence between
$(X_A,\sigma_A)$ and $(X_B,\sigma_B)$.
For $f \in C(X_B,{\mathbb{C}})$ and $x \in X_A$,
let us define
$\Psi_h(f)(x) $
 by the same formula as
\eqref{eq:Psihfx}.
We use the same notation 
$\Psi_h$ as the previous sections without confusions.
\begin{lemma}
Keep the above notations.
The map
$\Psi_h:  C(X_B,{\mathbb{C}}) \longrightarrow  C(X_A,{\mathbb{C}})$
is a continuous linear map such that
\begin{enumerate}
\renewcommand{\labelenumi}{(\roman{enumi})}
\item 
$\varphi \circ \Psi_h \in M(X_B,\sigma_B)$ 
for $\varphi \in M(X_A,\sigma_A)$.
\item 
$\varphi \circ \Psi_h \in M(X_B,\sigma_B)_+$ 
for $\varphi \in M(X_A,\sigma_A)_+$.
\item 
If in particular $[c_1] = [1]$ in $H^A$, we have  
$\varphi \circ \Psi_h \in P(X_B,\sigma_B)$ 
for $\varphi \in P(X_A,\sigma_A)$.
\end{enumerate}
\end{lemma}
\begin{proof}
(i)
Since $l_1, k_1 \in C(X_A,{\mathbb{Z}}_+)$, 
there exists $0< M \in {\mathbb{R}}$ such that 
$\sup_{x \in X_A}(l_1(x) + k_1(x)) \le M$, so that
it is easy to see that the inequality
$\| \Psi_h(f) \| \le M \| f \|$ for $f \in C(X_A,{\mathbb{C}})$
holds.
Hence 
$\Psi_h:  C(X_B,{\mathbb{C}}) \longrightarrow  C(X_A,{\mathbb{C}})$
is a continuous linear map.
We note that 
the same equality
as in  Lemma \ref{lem:cobdy}
holds for $f \in  C(X_B,{\mathbb{C}})$
by its proof.
For 
$\varphi \in M(X_A,\sigma_A)$
and
$f \in  C(X_B,{\mathbb{C}})$, 
it follows that
\begin{align*}
\varphi(\Psi_h(f \circ \sigma_B))
& = \varphi(\Psi_h(f \circ \sigma_B - f) + \Psi_h(f)) \\
& = \varphi(f \circ h \circ \sigma_A - f\circ h) + \varphi(\Psi_h(f)) \\
& = \varphi(\Psi_h(f))
\end{align*}
so that 
$\varphi \circ \Psi_h $ is $\sigma_B$-invariant.

(ii)
We next assume 
$\varphi  \in M(X_A,\sigma_A)_+$
a positive measure on $X_A$.
As 
$\bar{\Psi}_h({H}^B_+) \subset {H}^A_+$,
for
$f \in  C(X_B,{\mathbb{Z}}_+)$
there exist
$f_o \in  C(X_A,{\mathbb{Z}}_+),$
$g_o \in  C(X_B,{\mathbb{Z}})$
such that
$$
\Psi_h(f) = f_o + g_o - g_o\circ\sigma_A
$$
 so that
\begin{equation*}
\varphi(\Psi_h(f)) 
=\varphi( f_o + g_o - g_o\circ\sigma_A) 
=\varphi( f_o )\ge 0.
\end{equation*}
 Let us next consider  a nonnegative real valued function $f$ on $X_B$.
It is written 
$
f = \sum_{i=1}^n r_i \chi_{U_{\mu(i)}}
$
for some
$ 0 \le r_i \in {\mathbb{R}}$
and
$\mu(i) \in B_*(X_B), i=1,\dots,n$,
where
$\chi_{U_{\mu(i)}}$ is the characteristic function of the cylinder set
$U_{\mu(i)}$ for the word $\mu(i)$.
Since
$\Psi_h$ is linear and $\chi_{U_{\mu(i)}} \in C(X_B,{\mathbb{Z}}_+)$,
one has
$\Psi_h(\chi_{U_{\mu(i)}}) \ge 0$ as above so that
 \begin{equation*}
\varphi(\Psi_h(f)) 
= \sum_{i=1}^n r_i \varphi(\chi_{U_{\mu(i)}}) \ge 0
\end{equation*}
and hence we have
 $\varphi\circ\Psi_h \in M(X_A,\sigma_A)_+$.
 
 (iii)
If  
$[c_1] = [1] $ in $H^A$, we have
 \begin{equation*}
\varphi(\Psi_h(1)) 
= \varphi(c_1) = \varphi(1).
\end{equation*}
This implies that
$\varphi\circ\Psi_h \in P(X_A,\sigma_A)_+$
for 
$\varphi \in P(X_B,\sigma_B)_+$.
\end{proof} 

\begin{theorem}\label{thm:measure}
Suppose that
the  one-sided topological Markov shifts
$(X_A,\sigma_A)$ and $(X_B,\sigma_B)$
are continuously orbit equivalent 
through a homeomorphism
 $h:X_A\longrightarrow X_B$.
Then
the previousely defined isomorphisms
\begin{equation*}
\Psi_h: C(X_B,{\mathbb{Z}}) \longrightarrow C(X_A,{\mathbb{Z}}),
\qquad
\Psi_{h^{-1}}: C(X_A,{\mathbb{Z}}) \longrightarrow C(X_B,{\mathbb{Z}})
\end{equation*}
 extend to continuous linear maps between the Banach spaces
\begin{equation*}
\Psi_h: C(X_B,{\mathbb{C}}) \longrightarrow C(X_A,{\mathbb{C}}),
\qquad
\Psi_{h^{-1}}: C(X_A,{\mathbb{C}}) \longrightarrow C(X_B,{\mathbb{C}})
\end{equation*}
 which are inverses to each other such that 
\begin{enumerate}
\item
$\Psi_h$ (resp. $\Psi_{h^{-1}}$) maps
the $\sigma_A$ (resp. $\sigma_B$)-invariant regular Borel positive measures 
on $X_A$ (resp. $X_B$)
to 
the $\sigma_B$ (resp. $\sigma_A$)-invariant regular Borel positive measures 
on $X_B$ (resp. $X_A$).
\item
If in particular,
the class $[c_1]$ (resp. $[c_2]$) of the cocycle function 
$c_1$ (resp. $c_2$) is cohomologus to $1$ in $H^A$ (resp. $H^B$),
$\Psi_h$ (resp. $\Psi_{h^{-1}}$) maps
the $\sigma_A$ (resp. $\sigma_B$)-invariant 
regular Borel probability measures 
on $X_A$ (resp. $X_B$) to the $\sigma_B$ (resp. $\sigma_A$)-invariant 
regular Borel probability measures on $X_B$ (resp. $X_A$).
\end{enumerate}
 \end{theorem}
Hence 
the set of the shift-invariant regular Borel measures on the one-sided shift space is invariant under continuous orbit equivalence of one-sided topological 
Markov shifts.


\begin{thebibliography}{99}














\bibitem{BH}
{\sc M. Boyle and D. Handelman},
{\it Orbit equivalence, flow equivalence and ordered cohomology},
Israel J.\  Math.\
{\bf 95}(1996), 
pp.\ 169--210.

























\bibitem{CK}{\sc J. ~Cuntz and W. ~Krieger},
{\it A class of $C^*$-algebras and topological Markov chains},
 Invent.\ Math.\
 {\bf 56}(1980), pp.\ 251--268.





















\bibitem{Hu}
{\sc D. Huang},
{\it Flow equivalence of reducible shifts of finite type},
Ergodic Theory Dynam. Systems
{\bf 14}(1994), 
pp. \ 695--720.


\bibitem{Hu2}
{\sc D. Huang},
{\it Flow equivalence of reducible shifts of finite type and Cuntz--Krieger algebras},
J. Reine Angew. Math.
{\bf 462}(1995), pp. \ 185--217.










\bibitem{LM}{\sc D. ~Lind and B. ~Marcus},
{\it An introduction to symbolic dynamics and coding},
 Cambridge University Press, Cambridge
(1995).







\bibitem{MaPacific}
{\sc K. Matsumoto},
{\it Orbit equivalence of topological Markov shifts and Cuntz--Krieger algebras},
Pacific J.\ Math.\ 
{\bf 246}(2010), 199--225.













\bibitem{MaPAMS}
{\sc K. Matsumoto},
{\it  Classification of Cuntz--Krieger algebras by orbit equivalence of topological Markov shifts},
Proc. Amer. Math. Soc.
{\bf 141}(2013), pp.\ 2329--2342.




\bibitem{MM}
{\sc K. Matsumoto and H. Matui},
{\it Continuous orbit equivalence of topological Markov shifts 
and Cuntz-Krieger algebras},
to appear in Kyoto J. Math..




\bibitem{MatuiPLMS}
{\sc H. Matui}, 
{\it Homology and topological full groups of {\'e}tale groupoids on totally disconnected spaces},
 Proc. London Math. Soc. {\bf 104}(2012), 
 pp.\ 27--56.

\bibitem{MatuiPre2012}
{\sc H. Matui}, 
{\it Topologicasl full groups of one-sided shifts of finite type},
to appear in J. Reine Angew. Math..

\bibitem{PP}
{\sc W. Parry and M. Pollicott},
{\it Zeta functions and the periodic orbit structure of hyperbolic dynamics},
Ast{\'e}risque
 {\bf 187-188}(1990).


\bibitem{PS}
{\sc W. Parry and D. Sullivan},
{\it A topological invariant for flows on one-dimensional spaces}
Topology
 {\bf 14}(1975), pp.\ 297--299.




\bibitem{Po}
{\sc Y. T. Poon},
{\it A K-theoretic invariant for dynamical systems}, 
Trans.\  Amer.\ Math.\ Soc.\ 
{\bf 311}(1989), pp.\ 513--533.









\bibitem{Re}{\sc J. Renault},
{\it A groupod approach to $C^*$-algebras},
Lecture Notes in Math.  793, 
Springer-Verlag, Berlin, Heidelberg and New York (1980).



\bibitem{Ruelle1978}{\sc D. Ruelle},
{\it Thermodynamic formalism},
 Addison-Wesley, Reading (Mass.)
(1978).



\bibitem{Ruelle2002}{\sc D. Ruelle},
{\it Dynamical zeta functions and transfer operators},
Notice  Amer.\ Math.\ Soc. {\bf 49}(2002), pp.\ 175--193.





\bibitem{Ro}
{\sc M. R{\o}rdam},
{\it Classification of Cunzt-Krieger algebras},
 K-theory {\bf 9}(1995), pp.\  31--58.























\end{thebibliography}
\end{document}